\documentclass[11pt,reqno]{amsart}

\usepackage{amsthm}
\usepackage{amssymb}
\usepackage{mathrsfs}
\usepackage{amsmath}
\usepackage{latexsym}
\usepackage{graphicx,tikz}
\usepackage[all]{xy}
\usepackage{ytableau}
\usepackage{graphicx}
\usepackage{eucal}
\usepackage{enumerate}

\usepackage[hidelinks]{hyperref}

\usepackage{enumitem}

\newtheorem{thm}{Theorem}[section]
\newtheorem{prop}[thm]{Proposition}
\newtheorem{lemma}[thm]{Lemma}

\newtheorem{conjecture}[thm]{Conjecture}

\theoremstyle{definition}
\newtheorem{defn}[thm]{Definition}

\newtheorem{example}[thm]{Example}
\newtheorem{remark}[thm]{Remark}

\input{appendix-preamble}

\begin{document}

\title[Hodge--Tate hyperplane sections and semisimple quantum cohomology]{A-D-E diagrams, Hodge--Tate hyperplane sections and semisimple  quantum cohomology}

\author[S. Galkin]{Sergey Galkin}
\address{PUC-Rio, Departamento de Matem\'{a}tica, Rua Marqu\^{e}s de S\~{a}o Vicente 225, G\'{a}vea, Rio de Janeiro, Brazil}

\email{sergey@puc-rio.br}
\thanks{}
\author[N.C. Leung]{Naichung Conan Leung}
\address{The Institute of Mathematical Sciences and Department of Mathematics, The Chinese University
of Hong Kong, Shatin, Hong Kong}

\email{leung@math.cuhk.edu.hk}
\thanks{
 }

\author[C. Li]{Changzheng Li}
 \address{School of Mathematics, Sun Yat-sen University, Guangzhou 510275, P.R. China}
\email{lichangzh@mail.sysu.edu.cn}

\author[R. Xiong]{Rui Xiong}

\address{Department of Mathematics, University of Michigan, Ann Arbor, MI 48109, USA}
\email{xiongrui@umich.edu}

\makeatletter
\let\@wraptoccontribs\wraptoccontribs
\makeatother
\contrib[With an appendix by]{Pieter Belmans and Markus Reineke}

\begin{abstract}
It is known that the semisimplicity of quantum cohomology implies the vanishing of off-diagonal Hodge numbers (Hodge--Tateness).
We investigate which hyperplane sections of homogeneous varieties possess either of the two properties.
   We provide a new efficient criterion for non-semisimplicity of the small quantum cohomology ring of Fano manifolds 
that depends only on the Fano index and Betti numbers.
   We construct a bijection between Dynkin diagrams of types $A$, $D$ or $E$,
and complex Grassmannians with Hodge--Tate smooth hyperplane sections.
   By applying our criteria and using monodromy action, we completely characterize the semisimplicity of the small quantum cohomology of smooth hyperplane sections in the case of complex Grassmannians, and verify a conjecture of Benedetti and  Perrin in the case of (co) adjoint Grassmannians.
\end{abstract}

\maketitle

\section*{Highlights}
\begin{itemize}
\item
Hyperplane sections of complex Grassmannians $Gr(k,n)$ are Hodge--Tate
if and only if there is a Dynkin diagram of type $A_n$ ($k=1$), $D_n$ ($k=2$) or $E_n$ ($k=3$ and $n\leq 8$)
  obtained by adding a node adjacent to the Dynkin diagram of type $A_{n-1}$. Namely,
$$\text{\upshape (i)}
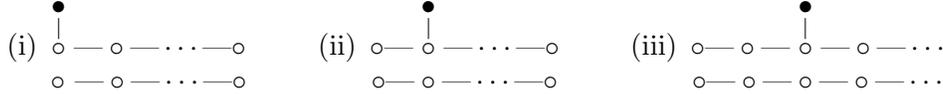
\begin{array}{c}
\stackrel{\stackrel{\displaystyle\bullet}{|}}
\circ\text{---}\circ\text{---}\cdots \text{---}\circ\\
\circ\,\,\text{---}\circ\text{---}\cdots \text{---}\circ
\end{array}\qquad
\text{\upshape (ii)}
\begin{array}{c}
\circ\text{---}
\stackrel{\stackrel{\displaystyle\bullet}{|}}\circ\text{---}\cdots \text{---}\circ\\
\circ\text{---}\circ\text{---}\cdots \text{---}\circ
\end{array}\qquad
\text{\upshape (iii)}
\begin{array}{c}
\circ\text{---}
\circ\text{---}
\stackrel{\stackrel{\displaystyle\bullet}{|}}\circ\text{---}\circ\text{---}\cdots\\
\circ\text{---}\circ\text{---}\circ\text{---}\circ\text{---}\cdots
\end{array}
$$

\item Quite often knowledge of Betti numbers and Fano index is sufficient
to witness non-semisimplicity of small quantum cohomology, even for Hodge-Tate Fano manifolds. Complete characterizations for the case of hyperplane sections of generalized Grassmannians are proposed in Conjecture \ref{conj: ssQHY}, with partial cases confirmed.
\item Quantum Lefschetz hyperplane principle admits user-friendly formulations,
      namely Proposition \ref{prop: QLP} after \cite{GaIr25},
      directly applicable by non-specialists in Gromov--Witten theory.
      We illustrate this advantage in the case of $Gr(3,8)$, an avatar of $E_8$.
\end{itemize}

\section{Introduction}
The  big quantum cohomology ring $BQH^*(X)$ of a Fano manifold $X$ encodes genus-zero Gromov--Witten invariants, and is canonically equipped with a Frobenius manifold structure. It is important to investigate the semisimplicity of $BQH^*(X)$. For instance by Givental's conjecture  \cite{Giv} proved by Teleman \cite{Tel}, all higher genus Gromov--Witten invariants of $X$   are determined by the genus-zero ones, provided that $BQH^*(X)$ is generically semisimple. The remarkable Dubrovin's conjecture \cite{Dub98} with clarification \cite{GMS15} on part 1, together with the refinement of its part 3 known as  Gamma conjecture II \cite{GGI} and independently formulated in \cite{Dub13, CDG}, is also concerned with the semisimplicity of the big quantum cohomology.

The  (small)   quantum cohomology ring $QH^*(X)=(H^*(X)\otimes \mathbb{C}[\mathbf{q}], \star)$ is a deformation of the classical cohomology $H^*(X)=H^*(X, \mathbb{C})$ by incorporating genus-zero, three-point Gromov--Witten invariants.
It is relatively more accessible than $BQH^*(X)$, and if $QH^*(X)$ is semisimple at some specialization of the quantum variables $\mathbf{q}$, then $BQH^*(X)$ is generically semisimple. Some important classes of Fano manifolds have  generically semisimple small quantum cohomology, including    toric Fano manifolds \cite{Bat, OsTy},      complete flag manifolds  \cite{Kos} and  (co)minuscule Grassmannians \cite{CMP}.
Nevertheless, there are some  examples with nice geometry such as most of (co)adjoint Grassmannians \cite{ChPe, PeSm} that have non-semisimple small quantum cohomology, whilst having its big quantum cohomology generically semisimple.
   It is known that semisimplicity of either big or small quantum cohomology $QH^*(X)$ puts some constraints on the classical cohomology $H^*(X)$. The first one was given by Bayer and Manin,   and was strengthened by Hertling,  Manin and   Teleman.
\begin{prop}[\cite{BaMa, HMT}]\label{prop:hodge}
    The even part of big quantum cohomology $BQH^{\rm ev}(X)$ is generically semisimple only if $H^*(X)$ is of Hodge--Tate type, i.e. $H^*(X)=\bigoplus_p H^{p,p}(X)$.
\end{prop}
\noindent There is another criterion for generic semi-simplicity of small quantum cohomology given by Chaput and Perrin \cite[Theorem 4]{ChPe} in terms of the degree of defining equations for a presentation of $H^*(X)$ when $X$ is of Picard number one.

Recall that the Fano index of $X$ is defined by
$$r=r_X := \max\{k\in \mathbb{Z}\mid {\frac{c_1(X)}{k}}\in H^2(X, \mathbb{Z})\}.$$
Define \emph{index-periodic even Betti numbers}
$$\tilde{b}_{\bar i} = \tilde{b}_{\bar i}(X) := \sum_{j\equiv i\bmod r_X} b_{2j}(X),$$
with indices $\bar i$ running over residue classes modulo $r_X$. As a very simple observation, we obtain the following theorem.
\begin{thm}\label{mainthm:ss}
The even quantum cohomology $QH^{\rm ev}(X)$ is generically semisimple only if both of the following hold.
\begin{enumerate}[label=\rm(\arabic*)]
  \item $\tilde{b}_{\bar i}(X) \leq \tilde{b}_{\overline{di}}(X)$ for all integer $i$ and $d$,
  \item $\tilde{b}_{\bar i}(X) = \tilde{b}_{-\bar i}(X)$ for all integer $i$.
 \end{enumerate}
\end{thm}
\noindent Here part (2)  is a direct consequence of part (1) by taking $d=r-1$.  We specify this property  as it is already very useful in many cases. For instance, we consider a generalized Grassmannian $G/P_k$, where $G$ is a connected complex simple Lie group and $P_k$ is the maximal parabolic subgroup that corresponds to the complement of the $k$-th simple root in the base of simple roots. Here we follow the label of the Dynkin diagram for $G$ as in   \cite{Bour}, see also \cite{Bel}.
Therein we can see that 9 (resp. 14) among all the 27 Grassmannians of exceptional Lie type   have semisimple (resp. non-semisimple)  small quantum cohomology. As a first application,
we will reprove in \textbf{Theorem \ref{thm:GrEF}} that    the non-semisimplicity of all the 14 known     cases but $E_8/P_4$   follows directly from Theorem \ref{mainthm:ss}  (2).

{
A second application is given in Appendix~\ref{sec:appendix}, by Pieter Belmans and Markus Reineke,
where it is shown in Theorem~\ref{theorem:non-semisimple}
that Theorem~\ref{mainthm:ss} is also strong enough to deduce the non-semisimplicity
of a certain infinite family of Kronecker moduli.
These are instances of moduli spaces of stable quiver representations,
which for other choices of the parameters include the usual Grassmannians,
which have semisimple quantum cohomology.
}

As one main aim of this paper, we investigate the   semisimplicity of $QH^*(Y)$ for a smooth hyperplane section $Y$ of a generalized Grassmannian $G/P_k$   (with respect to its minimal
embedding).
It is natural to  start with  $G$   of type $A$, i.e. the complex Grassmannian $A_{n-1}/P_k=Gr(k, n)=\{V\leq\mathbb{C}^n\mid \dim V=k\}$.  When  $k=1$, $Y=\mathbb{P}^{n-2}$ and   $QH^*(Y)$ is generically semisimple.
When $k=2$ and  $n=2m$ with $m>2$,  $Y\simeq C_m/P_2$   and $QH^*(Y)$ is   non-semisimple \cite{ChPe}.   When    $k=2$ and  $n=2m+1$ with $m>1$, $Y$ is a quasi-homogeneous variety and  $QH^*(Y)$ is   generically semisimple \cite{Pec, Per}.
When  $k>2$ and $n>3k$,  $H^*(Y)$ is not of Hodge--Tate type by  \cite[Theorem 3]{BFM21}.
As one main result of this paper, we provide the following complete characterization, which is  a combination of  \textbf{Theorem \ref{mainthm:HodgehyperplaneGr}} and \textbf{Theorem \ref{mainthm:QHhyperplaneGr}}.

\begin{thm}\label{mainthm:ttypeA}
Let $Y$ be a smooth hyperplane section of $Gr(k, n)\simeq Gr(n-k, n)$ where $n\geq 2k$.
\begin{enumerate}
    \item[{\upshape (1)}]  $H^*(Y)$ is of Hodge--Tate type if and only if  one of the following holds:
    $$ \qquad {\rm (i) }\,\, k=1; \quad {\rm (ii) } \,\, k=2;\quad    {\rm {(iii)} }\,\, k=3 \mbox{ and } n\in \{6, 7, 8\}.$$
 \item[{\upshape (2)}]
  $QH^*(Y)$ is generically semisimple if and only if one of the following holds:
$$\qquad {\rm (i) }\,\, k=1; \quad {\rm (ii) } \,\, k=2 \mbox{ and  } n\in \{4\}\cup \{2m+1\mid m\in \mathbb{Z}_{\geq 2}\}; \quad {\rm (iii) } \,\, k=3  \mbox{ and } n \in \{7, 8\}.$$
\end{enumerate}
\end{thm}

\begin{remark} \label{remark-ade}
The classification in Theorem \ref{mainthm:ttypeA} (1)
is in bijection with ADE Dynkin diagrams in the sense that the diagram obtained by adding a node adjacent to the $k$-th node of type $A_{n-1}$ is a simply-laced Dynkin diagram.
Moreover, by \cite{BFM, SaKi,Sco06}, the following are all equivalent:
\begin{enumerate}
      \item $H^*(Y)$ is of Hodge--Tate type.
      \item  $Y$ is locally rigid.
      \item The fundamental representation $\Lambda^k\mathbb{C}^n$ of $GL_n$ has a Zariski open orbit.
     \item The cluster algebra of $A_{n-1}/P_k$ is of finite type.
\end{enumerate}
However, we do not have a conceptual explanation for the equivalence.
\end{remark}

Besides the cases in Theorem \ref{mainthm:ttypeA} (1),
(co)adjoint Grassmannians $G/P_k$   provide another family of generalized Grassmannians such that the cohomology of a smooth hyperplane section $Y$ is of Hodge--Tate type, as shown by Benedetti and  Perrin \cite{BePe}.
Such  hyperplane sections $Y$  admit a uniform characterization. Following the terminology  in \cite{BePe} and restricting to the minimal-embedding, we list all the (co)adjoint Grassmannians in  Table \ref{tab: coadj}. Therein we further   assume $n\geq3$ for types $B_n,C_n$
and $n\geq4$ for type $D_n$; the lower-rank cases
are projective spaces or quadrics and are treated separately.
\begin{table}[hhhhhh]
 $$\def\c#1#2{\begin{array}{@{}c@{}}#1\\#2\end{array}}
            \begin{array}{|c|c|c|c|c|c|c|c|c|}\hline
            \text{type} &B & C & D & E_6 & E_7& E_8 & F_4 & G_2\\\hline\hline
            \c{\text{adjoint}}{\text{coadjoint}}&
            \c{B_n/P_2}{\text{--}} &
            \c{\text{--}}{C_n/P_2} &
            D_n/P_2 &
            E_6/P_2 &
            E_7/P_1 & E_8/P_8 &
            \c{F_4/P_1}{F_4/P_4} &
            \c{G_2/P_2}{G_2/P_1}\\\hline
            \end{array}$$
   \caption{Minimally embedded (Co)Adjoint Grassmannians $G/P_k$}\label{tab: coadj}
   \vspace{-0.5cm}
\end{table}

By Conjecture 1.10 (2) in \cite{BePe},  $QH^*(Y)$ is generically semisimple if and only if so is $QH^*(G/P_k)$. Now this conjecture holds true by combining the study of the semisimplicity of $QH^*(G/P_k)$ in \cite{ChPe, PeSm} and the following theorem.
\begin{thm}\label{mainthm: adcoad}
Let $Y$ be a smooth hyperplane section of $G/P_k$ of (co)adjoint  type.
Then  $QH^*(Y)$ is generically  semi-simple
if and only if
 $G/P_k$ is adjoint and not coadjoint or is isomorphic to a smooth quadric, i.e. $G/P_k\in \{B_n/P_2, C_n/P_1, F_4/P_1, G_2/P_2, G_2/P_1\}$.
 \end{thm}
\noindent

The semisimplicity of  $QH^*(Y)$ has been precisely determined   except for $Y$ in   $C_n/P_2$ or $D_n/P_2$ (see \cite[Theorem 1.11]{BePe}).
As another main result of this paper, we obtain the following, which  is a combination of  \textbf{Theorem  \ref{thm: type C}} and \textbf{Theorem \ref{thm: type D}} and fills in the last piece of the proof of Theorem \ref{mainthm: adcoad}.
\begin{thm}\label{mainthm:typeCD}
Let $Y$ be a smooth hyperplane section of $C_n/P_2$ or $D_n/P_2$.
 Then $QH^*(Y)$ is not semi-simple.
\end{thm}

 For $Y$ in $Gr(3, 6)=A_5/P_3$ or the coadjoint Grassmannians $C_n/P_2, F_4/P_4$, we achieve  the non-semisimplicity of $QH^*(Y)$ by a direct application of  Theorem \ref{mainthm:ss} (2). 

The monodromy  action plays a surprising role in the proof (resp. disproof) of the semisimplicity of quantum cohomology for $Y$ in $A_7/P_3$ (resp. $D_n/P_2$).
More precisely, a hyperplane section depends on a choice of global sections,
and thus can be viewed as a fiber of a universal family (see Sections 3.3 and 4.2 for precise descriptions).
The fundamental group of the smooth locus  acts on $H^*(Y)$ by monodromy, and preserves the quantum product.
We obtain the semisimplicity of (the radical part of) $QH^*(Y)$ in case $A_7/P_3$ by using Deligne's invariant cycle theorem. We  obtain the non-semisimplicity in case $D_n/P_2$  by carrying out a perverse sheaf-theoretic study of the universal family and applying Springer theory in geometric representation theory.
This will lead to an extra $\mathbb{Z}_2$-graded algebra structure of   the quantum cohomology, so that  Theorem \ref{mainthm:ss} (1) (strictly speaking, Lemma \ref{lemma: key}) can be applied.
 We remark that our proof works straightforwardly for $Y$ in a coadjoint Grassmannian of type $E$ as well.
We also notice that the monodromy action method has been used to study Gromov--Witten invariants \cite{Hu15, Mi19} before.

In  our proof of Theorem \ref{mainthm:ttypeA} (1), the key ingredient  is  the   combinatorial characterization  of nonvanishing cohomology by Snow \cite{Snow86} (see Proposition \ref{prop: Snow}) from the parabolic Borel--Weil Theorem by Bott \cite{Bott57}. The hypotheses on $(k, n)$ in  Theorem \ref{mainthm:ttypeA} (1) are equivalent to the following condition (see Lemma \ref{lemma: nk})
$$\dim Gr(k, n)=k(n-k)<2n=2 r_{Gr(k, n)}.$$
For any $G/P_k$ of (co)adjoint  type, the inequality $\dim G/P_k< 2 r_{G/P_k}$ holds by direct calculations, and $H^*(Y)$ is always of Hodge--Tate type \cite{BePe}.
These observations,  together with our  computations in examples of general Lie type, lead us to the  following conjecture\footnote{Conjecture \ref{conj: ssQHY} can be established as  a theorem  via an AI-assisted  extension of   our proof of Theorem \ref{mainthm:ttypeA} (1). We have chosen not to include this proof in the present paper, keeping all the mathematical results the entire work of  humans.}.
We remark that the inequality $\dim X< 2r_X$ also appeared in the early study of Fano complete intersections in projective spaces \cite{Bea, TiXu}.
Although the cases for $G/P_k$ of type $A$ or (co)adjoint type have been classified in Theorems \ref{mainthm:ttypeA} and \ref{mainthm: adcoad}, we include them in the statement below for completeness.

\begin{conjecture}\label{conj: ssQHY}
Let $Y$ be a smooth hyperplane section of $G/P_k$.
Then
    $H^*(Y)$ is of Hodge--Tate type if and only if $\dim G/P_k< 2 r_{G/P_k}$, namely $G/P_k$ is given by one of the cases:

    \begin{enumerate}[label=\rm(\alph*)]
        \item a complex Grassmannian in Theorem \ref{mainthm:ttypeA} (1);

        \item a (co)adjoint Grassmannian in Table \ref{tab: coadj},  $C_n/P_1$, $B_n/P_1$ or $D_n/P_1$;

        \item {\upshape(i)}   $C_3/P_3$,    {\upshape(ii)} $B_{n-1}/P_{n-1}\simeq D_n/P_n\simeq D_n/P_{n-1}$ for $n\in \{4,5,6,7\}$,
        {\upshape(iii)} $E_6/P_1\simeq E_6/P_6$ or
        {\upshape(iv)} $E_7/P_7$.
    \end{enumerate}  
 \end{conjecture}

\begin{prop}\label{cor:ofconj}
Let $Y$ be a smooth hyperplane section of $G/P_k$. Assume Conjecture \ref{conj: ssQHY} holds. 
Then    $QH^*(Y)$ is generically semisimple if and only if $G/P_k$ is given by one of the cases:
        \begin{enumerate}[label=\rm(\alph*)]
        \item a complex Grassmannian in Theorem \ref{mainthm:ttypeA} (2);

        \item  $B_n/P_1$,   $B_n/P_2$,   
$C_n/P_1$,   $D_n/P_1$, $F_4/P_1$, $G_2/P_1$ or $G_2/P_2$;
        \item 
        {\upshape(i)} $C_3/P_3$,    {\upshape(ii)} $B_{n-1}/P_{n-1}\simeq D_n/P_n\simeq D_n/P_{n-1}$ for $n\in \{4,5\}$.
     \end{enumerate}
\end{prop}

\begin{remark}
    There are in total 5 generalized Grassmannians $G/P_k$ with $\dim G/P_k=2 r_{G/P_k}$, given by
     $A_7/P_4$, $A_8/P_3=A_8/P_6$, $B_7/P_7=D_8/P_7=D_8/P_8, C_4/P_4$ and $C_4/P_3$.   They are not $N$-Calabi--Yau in the sense of Bernardara, Fatighenti and Manivel \cite{BFM21}. For a general hyperplane section $Y$ in these cases, the Hodge number   $h^{r_Y+1, r_Y}(Y)$ is nonzero.
     As pointed out by Laurent Manivel, such a nonzero Hodge number in the first 4 cases is the genus of an algebraic curve that can be attached to $Y$. See \cite{GSW, BBFM} and the references therein for the incorporation of higher genus curves with representations in Vinberg theory.
\end{remark}

Finally, let us describe a convenient and powerful variation \cite{GaIr25} of a quantum Lefschetz hyperplane principle
that originally organized our investigation for the cases $A_6/P_3$ and $A_7/P_3$,
and helped to lead it through. Information about quantum cohomology of homogeneous varieties
is relatively easy to obtain and organize, in particular tables of \cite{GG2},
based on Peterson's quantum Chevalley formula often turn out to be sufficient.
In turn, hyperplane sections of homogeneous varieties  usually lack sufficient homogeneity,
with a notable exception of (co)adjoint varieties, as observed in \cite{BePe}.
So ways to transfer  information  about quantum cohomology from an ambient space
to its hypersurface or the other way around are valuable. Quantum Lefschetz Hyperplane Principles
provide a variety of such ways. Some of them more obviously can be related to classical
Lefschetz hyperplane theorems, and others (e.g. some spectral formulations below) may
look very different from classical counterparts, but more useful in practice.
One feature that may look surprising in contrast to the classical case,
is that in many non-trivial situations the passage from ambient space to hyperplane section
is invertible, up to some minor constants fitting.

Now we let $Y$ be a smooth ample  hypersurface of $X$ with natural inclusion $j: Y\hookrightarrow X$.  The genus-zero Gromov--Witten theory of $Y$ can be related with the (twisted) Gromov--Witten theory of $X$ by the quantum Lefschetz principle theorems of  \cite{Kim, Lee, CoGi}.
 In these works it is mainly phrased as a relation either between the virtual fundamental classes of   moduli spaces of stable maps  or between Givental's $J$-functions when passing from an ambient space $X$ to its hypersurface $Y$.
 There have   been various extensions to other situations, such as \cite{CCIT, Tse, IMM16}, and we refer to \cite{OhTh} and references therein for more progress in recent years.
 Note that the induced ring homomorphism
 $j^*: H^*(X)\to H^*(Y)=H^*_{\rm amb}(Y)\oplus H^*_{\rm prim}(Y)$ has image $j^*(H^*(X))=H^*(Y)_{\rm amb}$.  All the various versions of the aforementioned  quantum Lefschetz principle relate the information of $X$ to  the ambient part of the corresponding information of $Y$.
 By the classical Lefschetz hyperplane theorem, the restriction $j^*|_{H^m(X)}:H^m(X)\to H^m(Y)$ is an  isomorphism of abelian groups for $0\leq m< \dim Y$.

These original quoted versions of quantum Lefschetz principle turn out to be cumbersome for some applications in practice,
especially when important information is encoded differently.
See \cite[Section 6.6]{Gol07} for a formulation that relates quantum $D$-modules,
\cite[Section 7.3]{GoMa} for a relevant example that shows its convenience.
A version of a quantum Lefschetz principle on the level of quantum $D$-modules was
 in \cite[Corollary 1.2]{IMM16}.

A statement of a quantum Lefschetz hyperplane principle that directly relates $QH^*(X)$ with $QH^*(Y)$ on the level of algebras
is being developed by Galkin and Iritani. For the purposes of this article
the preliminary form available in \cite{GaIr25} would suffice.

For simplicity, here we restrict to Fano manifolds $X$ of Picard number $1$.
In this case, $QH^*(X)=H^*(X)\otimes \mathbb{C}[q_X]$ is a $\mathbb{C}[q_X]$-module, and we consider the linear operator $\hat \alpha$ on  $QH^\bullet(X):=QH^*(X)|_{q_X=1}$  induced by the quantum multiplication: $\beta\mapsto (\alpha\star \beta)|_{q_X=1}$.
Make $H^*(X)=\bigoplus_{\bar i}\mathcal{A}^{\bar i}$ into  a $\mathbb{Z}_r$-graded vector space, where
$$\mathcal{A}^{\bar i}:=\bigoplus_{j\equiv i\bmod r} H^{2j}(X).$$
Note, in particular, that periodic Betti numbers $\tilde{b}_{i} = \dim \mathcal{A}^{\bar i}$ are dimensions of graded pieces.
Define $Rad(X):=\{ \alpha\in H^*(X)\mid (c_1(X))^{\star m} \star \alpha=0 \mbox{ for some } m>0\}$, and consider its orthogonal complement $$\mathcal{A}^0_{\perp}(X):=\{\gamma \in \mathcal{A}^0(X)\mid \int_X\alpha \cup \gamma=0 \mbox{ for any } \alpha \in Rad(X)\}.$$
For $\dim Y>2$, by Lefschetz hyperplane theorem, the induced map $ \mathrm{Pic}(X) \to \mathrm{Pic}(Y)$
is an isomorphism of free cyclic groups $\mathrm{Pic}(X) = \mathbb{Z} \mathcal{O}_X(1)$ and $\mathrm{Pic}(Y) = \mathbb{Z} \mathcal{O}_Y(1)$
that sends an ample generator $\mathcal{O}_X(1)$ to an ample generator $\mathcal{O}_Y(1)$.
Denote by $h_X = c_1(\mathcal{O}_X(1))$ and $h_Y = j^* h_X = c_1(\mathcal{O}_Y(1))$,  so that $c_1(X)=r_X h_X$ and $c_1(Y)=r_Y h_Y$.

\begin{prop}[Quantum Lefschetz hyperplane; Galkin--Iritani \protect{\cite{GaIr25}}]\label{prop: QLP}
Let $j : Y\hookrightarrow X$ be
the natural inclusion
of a smooth Fano hypersurface $Y$
in a Fano manifold $X$ of Picard number 1.
Assume  $\dim Y >2$.
   \begin{enumerate}[label=\rm(\arabic*)]
    \item There is a   ring homomorphism $j^*_{q}: QH^*(X)\to QH^*(Y)$ that fits into a commutative diagram\\
  {}\hspace{2cm}\xymatrix{
   QH^*(X) \ar[d]_{\pi_X} \ar[r]^{j^*_q}
                & QH^*(Y) \ar[d]^{\pi_Y}   \\
 QH^*(X)/(q_X)= H^*(X)  \ar[r]^{j^*}
                & H^*(Y)=QH^*(Y)/(q_Y)            }\\
 and $j_q^*(q_X)\neq 0$. Moreover,
  $j^*_q(q_X) = q_Y h_Y$ if both $r_X - 1 = r_Y$ and $r_Y > 1$ hold.
     \item  When $r_Y>1$, the set of nonzero eigenvalues of $\widehat{h_X}^{r_X}$  coincides with that of $\widehat{h_Y}^{r_Y}$,
	    {up to dilation by a constant if $r_X-r_Y>1$.}
        \end{enumerate}
\end{prop}

\begin{remark}\label{rmk:forQL}
The stability of spectra in part (2) of the above proposition could have been observed independently more than once.
We know it was observed and used in early 2000s by mirror symmetry research group in Moscow,
that included Golyshev, Galkin, and Przyjalkowski, not later than 2004,
and is implicitly used in supplementary materials \cite{GG2} for \cite{GG}.
Some of these developments were announced in Golyshev's report \cite{Gol08} on spectra and strains,
including various definitions of spectra, and a claim of spectral stability. We note that the case $r_Y=1$ will further require a shift of Spec($\widehat{h_Y}$) by a constant.
\end{remark}

\begin{remark} In \cite{GaIr25}, Galkin and Iritani studied the relationship between $QH^*(X)$ and the Euler-twisted quantum cohomology $QH^*_{\rm tw}(X, L)$ (twisted by a nef line bundle $L$), and showed a ring homomorphism  $QH^*_{\rm tw}(X, L)\to QH^*(Y)$. As a consequence, they can obtain the current   Proposition \ref{prop: QLP}. In specific cases,   Proposition \ref{prop: QLP} can be verified more directly. Indeed, the cases of Fano complete intersections in projective spaces  follow immediately from  \cite[Corollaries 9.3 and 10.9]{Giv96}   and
   \cite[Lemmas 3.2 and 4.2]{Ke}. In Theorem \ref{thm: conjtrueforGr}, we also provide a direct verification for the cases of hyperplane sections in  $A_6/P_3$ or $A_7/P_3$.
\end{remark}

We may further compare $\mathcal{A}_{\perp}^0(X)$ and $\mathcal{A}_{\perp}^0(Y)$ as follows.
\begin{conjecture}\label{conj: QLP} With the same notation in Proposition \ref{prop: QLP},
there is a natural \emph{injective morphism} of algebras
$\mathcal{A}_{\perp}^0(X)\longrightarrow \mathcal{A}_{\perp}^0(Y)$
that intertwines operators $\widehat{h_X}^{r_X}$ and $\widehat{h_Y}^{r_Y}$, up to the normalization described in
Proposition \ref{prop: QLP}(2) and Remark \ref{rmk:forQL}
when $r_X-r_Y>1$ or $r_Y=1$.
\end{conjecture}
 
Conjecture \ref{conj: QLP}   tells us that    the semisimplicity of $QH^*(X)$ could be related to the semisimplicity of the subalgebra of  $QH^*(Y)$ generated by $\mathcal{A}_{\perp}^0(Y)$ and $c_1(Y)$. This could be useful when $Rad(Y)$ is of small dimension. Indeed, we were guided by this philosophy  when investigating $QH^*(Y)$ for $Y$ in  $A_6/P_3$ or $A_7/P_3$. We succeeded in verifying  all the expected properties, providing evidence for Conjecture \ref{conj: QLP} and achieving the semisimplicity of $QH^*(Y)$.

 This paper is organized as follows. In Section 2, we provide necessary conditions for the semisimplicity of small quantum cohomology.  In Section 3, we completely characterize the semisimplicity of  $QH^*(Y)$ for smooth hyperplane sections $Y$ in $A_{n-1}/P_k$.
 Finally in Section 4, we show the non-semisimplicity  of $QH^*(Y)$  for $Y$ in   $C_n/P_2$ or $D_n/P_2$.

\subsection*{Acknowledgements}
The authors would like to thank Pieter Belmans,  Peter L. Guo, Jianxun Hu, Xiaowen Hu, Hiroshi Iritani, Hua-Zhong Ke, Allen Knutson, Laurent Manivel, Jiayu Song, Mingzhi Yang, and Zhihang Yu   for helpful discussions. C.~Li is supported   by the National Key R \&  D Program of China No.~2023YFA1009801.
S.~Galkin is supported by CNPq grants PQ 315747 and PQ 308303, and Coordena\c{c}\~{a}o de Aperfei\c{c}oamento de Pessoal de N\'{i}vel Superior-Brasil (CAPES)-Finance Code 001. N.C. Leung is
  substantially supported by grants from the Research Grants Council of the
Hong Kong Special Administrative Region, China (Project No. CUHK14305923 and CUHK14306322).
R. Xiong was partially supported by the NSERC Discovery
grant RGPIN-2022-03060, Canada.

\section{Criteria for the semisimplicity}

\subsection{Semisimple commutative algebra}
Let   $\mathcal{A}$ be a  finite-dimensional commutative (unital) $\mathbb{C}$-algebra.  The algebra $\mathcal{A}$ is called semisimple if for any $\alpha\in \mathcal{A}$, the induced linear operator $\hat \alpha: \mathcal{A}\to \mathcal{A}; \beta\mapsto \alpha\cdot \beta$ is semisimple; or equivalently, every nonzero $\alpha$ is not  nilpotent.
 Suppose that  $\mathcal{A}=\bigoplus_{\chi\in M} \mathcal{A}^\chi$ is equipped with an $M$-graded algebra structure, where  $M$ is an abelian group.

 \begin{lemma}\label{lemma: key}
      Suppose that  there exists $(d, \chi)\in \mathbb{Z}_{>0}\times M$ such that $\dim \mathcal{A}^{\chi}> \dim \mathcal{A}^{d \chi}$. Then there exists $\varepsilon \in \mathcal{A}\setminus \{0\}$ such that $\varepsilon^d=0$.
 \end{lemma}
   \begin{proof}
       Take a basis $\{e_i\}_{i=1}^n$ (resp. $\{\hat e_j\}_{j=1}^m$) of $\mathcal{A}^\chi$ (resp. $\mathcal{A}^{d\chi}$). Then $0=\varepsilon^d=(\sum_{i=1}^n a_ie_i)^d=\sum_{j=1}^m f_j(\mathbf{a}) \hat e_j$ holds if and only if $\mathbf{a}\in \mathbb{C}^n$ is a common root of the
       $m$  homogeneous polynomials $f_j(y_1, \cdots, y_n)$ of degree $d>0$. Since $n>m$, there exists a nonzero common root.
   \end{proof}

 \begin{lemma}\label{lemma: AtoAy} Take any $m>0$ and any $\alpha\in \mathcal{A}$ with $\hat \alpha$ invertible.   Then  $\mathcal{A}$ is semisimple if and only if $\mathcal{A}[y]/(y^m-\alpha)$ is semisimple.
 \end{lemma}
 \begin{proof}
 If $\mathcal{A}$ is semisimple, we have the decomposition $\mathcal{A}=\bigoplus_i A_i$ into one-dimensional subalgebras $A_i$. The map $\alpha\mapsto \alpha_i$ by taking the $A_i$-component of $\alpha$ induces a ring homomorphism $\mathcal{A}[y]\to A_i[y]$, and it further induces a ring isomorphism $\mathcal{A}[y]/(y^m-\alpha) \cong \bigoplus_i A_i[y]/(y^m-\alpha_i)$. As each $A_i$ is isomorphic to the complex field, the algebra $A_i[y]/(y^m-\alpha_i)$ is semisimple if and only if every root of $y^m-\alpha_i$ is of multiplicity one. Note that $\alpha$ is invertible if and only if every component $\alpha_i$ is invertible. Hence,
 $y^m-\alpha_i$ does not have multiple roots.

     Conversely, the subalgebra $\mathcal{A}$ of the semisimple algebra  $\mathcal{A}[y]/(y^m-\alpha)$ is semisimple.
 \end{proof}

    \subsection{Quantum cohomology}  We refer to \cite{CoKa} for more details of Gromov--Witten theory.
   Throughout the paper, we assume  $X$ to be a Fano manifold with even cohomology only.  Let  $\overline{\mathcal{M}}_{0, m}(X, \mathbf{d})$ denote the moduli space of stable maps to $X$ of degree $\mathbf{d}\in H_2(X, \mathbb{Z})$,
 and ${\rm ev}_i:\overline{\mathcal{M}}_{0, m}(X, \mathbf{d})\to X$ denote the $i$th   evaluation map.     For $\gamma_1, \cdots, \gamma_m\in H^*(X)$, we consider the  genus-zero,
$m$-point Gromov--Witten invariant   defined by
\begin{equation}
    \langle \gamma_1, \cdots, \gamma_m\rangle^X_{\mathbf{d}}:=\int_{[\overline{\mathcal{M}}_{0,m}(X, \mathbf{d})]^{\rm vir}} {\rm ev}_1^*(\gamma_1)\cup\cdots \cup {\rm ev}_m^*(\gamma_m).
\end{equation}
Here  the virtual fundamental class $[\overline{\mathcal{M}}_{0,m}(X, \mathbf{d})]^{\rm vir}$ lies in $H_{2{\rm expdim}}(\overline{\mathcal{M}}_{0,m}(X, \mathbf{d}), \mathbb{Q})$, where \begin{equation}
    {\rm expdim}= \dim X+\int_{\mathbf{d}}c_1(X)+m-3.
\end{equation}

Take a basis $\{[C_1], \cdots, [C_b]\}$ of the Mori cone $\overline{\rm NE}(X)$ of effective curve classes. (Here we have assumed that the effective curve semigroup is free for notational simplicity; the proof of Theorem \ref{mainthm:ss} is unchanged in general.) Each generator $[C_i]$ associates with an indeterminate $q_i$. For $\mathbf{d}= \sum_j d_j[C_j]$, we denote $\mathbf{q}^{\mathbf{d}}:= \prod_j q_j^{d_j}$.
The (small) quantum cohomology $QH^*(X)=(H^*(X)\otimes \mathbb{C}[q_1, \cdots, q_b], \star)$ is an associative commutative  algebra with unit $1\in H^0(X) \otimes \mathbb{C}$, with the quantum product defined by
$$\alpha_i\star \alpha_j:=\sum_k\sum_{\mathbf{d}\in\overline{\rm NE}(X)} \langle \alpha_i, \alpha_j, \alpha_k^\vee\rangle_{\mathbf{d}}^X \alpha_k \mathbf{q}^{\mathbf{d}};  $$
\noindent Here $\{\alpha_i\}_i$ denotes a basis of $H^*(X)$, and
$\{\alpha_i^\vee\}_i$ denotes its dual basis with respect to Poincar\'e pairing: $(\alpha_i, \alpha_j^\vee)_X=\int_{[X]} \alpha_i \cup \alpha_j^\vee =\delta_{i, j}$.
The quantum cohomology    $QH^*(X)$ is naturally a $\mathbb{Z}$-graded algebra with respect to the grading
 $$\deg q_i:=\int_{[C_i]}c_1(X),\qquad \deg \alpha:=j,\quad \forall \alpha\in H^{2j}(X)\setminus\{0\}.$$

 \bigbreak

\begin{proof}[Proof of Theorem \ref{mainthm:ss}]
Note that $r$ divides $\deg q_1, \cdots, \deg q_b$. Thus for any specialization $\mathbf{q}=\eta\in \mathbb{C}^b$,
the $\mathbb{Z}$-graded algebra $QH^*(X)$ naturally induces an $\mathbb{Z}_r$-graded algebra
$(H^*(X)=QH^*(X)|_{\mathbf{q}=\eta}=\bigoplus_{\bar i\in \mathbb{Z}_r}\mathcal{A}^{\bar i}, \star_{\eta})$,
 where $\mathcal{A}^{\bar i}=\bigoplus_{\bar j=\bar i} H^{2j}(X)$.

If  $QH^*(X)$ is semisimple at some   $\mathbf{q}=\eta$, then
 $(H^*(X), \star_{\eta})$ has no nonzero nilpotent element, and hence statement (1) follows from Lemma \ref{lemma: key}.

 Consequently, for any $i$, we have $\dim \mathcal{A}^{\bar i}\leq \dim \mathcal{A}^{(r-1)\bar i}= \dim\mathcal{A}^{-\bar i}$. Applying it to $-i$, we obtain $\dim \mathcal{A}^{-\bar i}\leq  \dim\mathcal{A}^{\bar i}$. Hence, we    have $\dim \mathcal{A}^{\bar i}=   \dim\mathcal{A}^{-\bar i}$. That is, statement (2) holds.
\end{proof}

 \subsection{First applications} Given a Lie type $\mathcal{D}_n$ and an integer $1\leq k\leq n$, we denote by $\mathcal{D}_n/P_k$
 the quotient of a simply-connected, complex simple Lie group $G$ of Lie type $\mathcal{D}_n$ by the maximal parabolic subgroup $P_k$ of $G$ that corresponds to the subset $\Delta\setminus\{\alpha_k\}$.
 Here $\Delta=\{\alpha_1, \cdots, \alpha_n\}$ is a base of simple roots for $G$ with the same ordering as in \cite{Bel}, and  $\mathcal{D}_n/P_k$ is called a generalized  Grassmannian (of type $\mathcal{D}_n$). In particular, the Grassmannian of type $A_{n-1}$,
  $$A_{n-1}/P_k=\{V\leq \mathbb{C}^n\mid \dim V=k\}=: Gr(k, n),$$
  is known as a complex Grassmannian.
 There are in total 27 Grassmannians of exceptional Lie type: 9 of which are known to have semisimple quantum cohomology, 14 are known to have non-semisimple quantum cohomology, and the other 4 cases are unknown \cite{Bel}. Here we provide a proof for 13 non-semisimple cases by Theorem \ref{mainthm:ss}.
 Unfortunately, the remaining non-semisimple case $E_8/P_4$ cannot be checked by Theorem \ref{mainthm:ss}, nor does it give any obstruction for the 4 unknown cases $E_7/P_2$, $E_7/P_4$, $E_7/P_5$, $E_8/P_6$.

\begin{thm}\label{thm:GrEF}
  For any $X\in\{E_6/P_2,  E_6/P_4, E_7/P_1, E_7/P_3, E_7/P_6, E_8/P_1, E_8/P_2, E_8/P_3, E_8/P_5,$ $ E_8/P_7, E_8/P_8, F_4/P_3, F_4/P_4 \}$, the   quantum cohomology $QH^*(X)$  is not semisimple.
\end{thm}

\begin{proof} Simply denote $r=r_X$ and $b_{2j}=b_{2j}(X)$. Then we can read off  the data of $\dim X, r, b_{2j}$  of each $X$ from \cite{Bel} directly, where $b_{2j}=0$ if $j>\dim X$. Then we can calculate $\dim \mathcal{A}^{\pm \bar 4}$ for $F_4/P_4$ and $\dim \mathcal{A}^{\pm \bar 1}$ for the other 12 cases.
For instance for $E_7/P_6$, we have
$\dim \mathcal{A}^{\bar 1}=b_{2}+b_{2\cdot 14}+b_{2\cdot 27}+b_{2\cdot 40}=1+26+29+2=58$, and $\dim \mathcal{A}^{-\bar 1}=b_{2\cdot 12}+b_{2\cdot 25}+b_{2\cdot 38}+b_{2\cdot 51}=21+34+4+0=59$.
 
\[
\begin{array}{|c|c|c|c|c|c|c|}
  \hline
  X & \dim X & r & i & m & \dim \mathcal{A}^{\bar i}=\sum\limits_{j=0}^m b_{2(i+jr)} & \dim \mathcal{A}^{-\bar i}=\sum\limits_{j=0}^m b_{2(r-i+jr)}   \\
 \hline \hline
 E_6/P_2 & 21 & 11 & &1 & 6 &7\\ \cline{1-3}\cline{5-7}
  E_6/P_4 & 29 & 7 & &4 & 102 &104\\ \cline{1-3}\cline{5-7}
   E_7/P_1 & 33 & 17 & &1 & 7 &8\\ \cline{1-3}\cline{5-7}
    E_7/P_3 & 47 & 11 & &4 & 183 &184\\ \cline{1-3}\cline{5-7}
   E_7/P_6 & 42 & 13 & &3 & 58 &59\\ \cline{1-3}\cline{5-7}
    E_8/P_1 & 78 & 23 & & 3 & 94 &95\\ \cline{1-3}\cline{5-7}
    E_8/P_2 & 92 & 17 & 1 & 5 & 1016 &1017\\ \cline{1-3}\cline{5-7}
    E_8/P_3 & 98 & 13 & & 7 & 5317 &5318\\ \cline{1-3}\cline{5-7}
     E_8/P_5 & 104 & 11 & & 10 & 21993 &21992\\ \cline{1-3}\cline{5-7}
      E_8/P_7 & 83 & 19 & & 4 & 354 &355\\ \cline{1-3}\cline{5-7}
       E_8/P_8 & 57 & 29 & & 1 & 8 &9\\ \cline{1-3}\cline{5-7}
          F_4/P_3 & 20 & 7 & & 2 & 13 &14\\ \hline
          F_4/P_4 & 15 & 11 & 4 & 2 & 3 &2\\ \hline
\end{array}
\] 
As from the table, none of the 13 cases satisfies $\dim \mathcal{A}^{\bar i}=\dim\mathcal{A}^{-\bar i}$ for the given $i$. Hence, none of them has semisimple quantum cohomology by Theorem \ref{mainthm:ss} (2).
 \end{proof}

\section{Smooth hyperplane sections of $Gr(k, n)$}
In this section, we let   $X=Gr(k, n)=\{V\leq \mathbb{C}^n\mid \dim V=k\}$, which is a closed subvariety in $\mathbb{P}^{{n\choose k}-1}$ via the Pl\"ucker embedding. The intersection of $X$ with a general hyperplane of $\mathbb{P}^{{n\choose k}-1}$ gives a smooth hyperplane section $Y$ of the complex Grassmannian $X$.
We will investigate the semisimplicity of $QH^*(Y)$.  Since $Gr(k, n)\cong Gr(n-k, n)$, we can always assume $n\geq 2k$.

\subsection{Characterization  of $H^*(Y)$ of Hodge--Tate type}
Due to Proposition \ref{prop:hodge}, we start with the study of the Hodge diamond of $Y$, which   has its own interest  in classical algebraic geometry. One key ingredient is the   combinatorial characterization  of nonvanishing cohomology by Snow \cite{Snow86} from the parabolic Borel--Weil Theorem by Bott \cite{Bott57}.

Denote by $\mathcal{P}_{k, n}:=\{\lambda=(\lambda_1, \cdots, \lambda_k)\in \mathbb{Z}^k\mid n-k\geq \lambda_1\geq \cdots \geq \lambda_k\geq 0\}$. Denote $|\lambda|:=\sum_{i}\lambda_i$.
The hook length of a cell  in the Young diagram of a partition $\lambda$   is defined to be the total number of cells which are either directly to the right or directly below the  cell   together with the cell.

\begin{prop}
[\protect{\cite[Section 3.1 (2)]{Snow86}}]
\label{prop: Snow}
For $\ell\geq 0$, we have
$$H^j(Gr(k,n),\Omega^p(\ell))\neq 0$$
if and only if there exists  $\lambda\in\mathcal{P}_{k, n}$ of $p$ cells with no cell of hook length $\ell$ such that
 $j$ equals the number of cells in $\lambda$ of hook length larger than $\ell$.
\end{prop}

\begin{remark}
Recall a smooth projective variety $X$ is said to satisfy Bott vanishing if
$$H^{j}(X,\Omega_X^p\otimes L)=0\qquad j>0,p\geq 0$$
for any ample line bundle $L$.
However, Bott vanishing fails for complex Grassmannians other than projective spaces, see \cite[Section 4.3]{BTLM}. We refer to \cite{Be25, Fon} for the failure of Bott vanishing for certain generalized Grassmannians of general Lie type.
\end{remark}

We call a partition $\lambda$ a \textit{$j$-core partition}, if $j$ does not appear as the hook length of cells in $\lambda$.
We first assume Lemma \ref{lem:comblemm},  which  will be proved  in Section \ref{section:lemmaproof} in a purely combinatorial way.

\begin{lemma}\label{lem:comblemm}
For $3\leq k\leq n/2$, there exists $(\lambda, i)\in  \mathcal{P}_{k, n}\times [1, n-1]$ such that $|\lambda|\geq k(n-k)-i$ and $\lambda$ is an $(n-i)$-core partition if and  only if
$(k,n)\in \{(3,6),(4,8),(3,9)\}$.
 \end{lemma}

\begin{example}\label{ex:kni}
Below are    partitions $(3,2,1),(4,3,2,1),  (6,4,2), (4,4,2,2)$, whose cells are labeled by their hook length.  For example, the partition  $(6,4,2)$ has 12 cells  with no cell of hook length $3$, and has 6 cells of hook length  larger than 3. Thus  $H^{6}(Gr(3,9),\Omega^{12}(3))\neq 0$ by Proposition \ref{prop: Snow}.
    $$
\begin{array}{c@{\qquad }c@{\qquad }c@{\qquad }c@{\qquad }c}
\lambda &
\begin{array}{|c|c|c|}\hline
5&3&1\\\hline
3&1\\\cline{1-2}
1\\\cline{1-1}
\end{array}&
\begin{array}{|c|c|c|c|}\hline
7&5&3&1\\\hline
5&3&1\\\cline{1-3}
3&1\\\cline{1-2}
1\\\cline{1-1}
\end{array}
&
\begin{array}{|c|c|c|c|c|c|}\hline
8&7&5&4&2&1\\\hline
5&4&2&1\\\cline{1-4}
2&1\\\cline{1-2}
\end{array}
&
\begin{array}{|c|c|c|c|}\hline
7&6&3&2\\\hline
6&5&2&1\\\hline
3&2\\\cline{1-2}
2&1\\\cline{1-2}
\end{array}\\\\
(k,n,i)&
(3,6,4)&(4,8,6)&(3,9,6)&(4,8,4)
\end{array}$$
As we will see from the proof of Lemma \ref{lem:comblemm}, they are the only $(n-i)$-core partitions.
\end{example}

\begin{lemma}\label{lemma: nk}
    For  $n\geq 2k$,  $k(n-k)>2n$ holds if and only if one of the following holds:
{\upshape $$
\mbox{(i) } k\geq 5;  \qquad \mbox{(ii) }
k=4 \text{ and }n\geq 9; \qquad \mbox{(iii) }
 k=3 \text{ and } n\geq 10.
$$
}
\end{lemma}

\begin{proof}
    For $k\geq 5$, $k(n-k)-2n=(k-2)n-k^2\geq (k-2)2k -k^2=k^2-4k>0$. For $k\leq 4$, by direct calculation,
 $k(n-k)>2n$ if and only if   $(k=3, n\geq 10)$ or $(k=4, n\geq 9)$ holds.
\end{proof}

\begin{lemma}\label{lem:vanishing}
Assume $\dim X>2n$.
Then for any $0\leq p\leq n$, we have
$$H^{s}(X,\Omega^{\dim X-j}_X(p-j))=0\qquad
\text{for all $s$ and $1\leq j<p$}.$$
\end{lemma}

\begin{proof}
Assume $H^{s}(X,\Omega^{\dim X-j}_X(p-j))\neq 0$ for some $(s, j, p)$. Then by Proposition \ref{prop: Snow}, there exists $\lambda\in \mathcal{P}_{k, n}$ with $|\lambda|=k(n-k)-j$, having no cell of hook length $p-j$ and having $s$ cells of hook length larger than $p-j$. Since $1\leq p-j<n$, then for $i:=n-p+j\in [j, n-1]$,  $|\lambda|\geq k(n-k)-i$ and $\lambda$ is an $(n-i)$-core partition. Thus by  Lemma \ref{lem:comblemm},
 $(k, n)\in \{(3, 6),  (3, 9), (4, 8)\}$. This contradicts  the assumption $\dim X=k(n-k)>2n$.
\end{proof}

We have the following proposition, whose  proof is a refined version of Griffiths' theory (\cite{Griff}, see also \cite[Section 6.1.2]{Voisin}) with more precise control of cohomology vanishing. A similar argument appears in \cite[Theorem 2.2]{DV10}.

\begin{prop}\label{prop: nonhodge}
Assume $k(n-k)>2n$.
Then
$$h^{k(n-k)-p,p-1}(Y)=\begin{cases}
0, & p<n,\\
1, & p=n.
\end{cases}$$
In particular, $H^*(Y)$ is not of Hodge--Tate type.
\end{prop}

\begin{proof}
Let $U=X\setminus Y$ be the complement of $Y$.
Let $\Omega^p_X(\log Y)$ be the sheaf of logarithmic $p$-forms, i.e. meromorphic differential $p$-forms $\phi$ such that $\phi$ and $d\phi$ both have a pole of order at most $1$ along $Y$.
 By the proof of \cite[Theorem 6.5]{Voisin}, the Hodge filtration of $H^{\dim X}(U,\mathbb{C})$ is given by
$$F^{\dim X-p+1}H^{\dim X}(U,\mathbb{C})
=H^{p-1}(X,\Omega_X^{\dim X-p+1}(\log Y)_{\rm{closed}}).$$
Note that by definition, $\Omega_X^{\dim X-p+1}(\log Y)_{\rm{closed}}=\Omega_X^{\dim X-p+1}(Y)_{\rm{closed}}$.
We have the following exact sequence for $p<\dim X$
$$0\to 
\Omega_X^{\dim X-p+1}(Y)_{\rm{closed}}
\stackrel{\subset}\to 
\Omega^{\dim X-p+1}_X(Y)
\stackrel{d}\to 
\Omega^{\dim X-p+2}_X(2Y)
\stackrel{d}\to 
\cdots 
\stackrel{d}\to 
K_X(pY)
\to 
0.
$$
By the vanishing in Lemma \ref{lem:vanishing}, we can shift the degree, and hence for $1\leq p\leq n$,
$$F^{\dim X-p+1}H^{\dim X}(U)\cong H^0(X,K_X(pY)).$$
Then by the assumption $\dim X>2n$, we have
$$F^{\dim X-p}H^{\dim X-1}(Y)\cong H^0(X,K_X(pY)).$$
It has dimension $0$ when $p<n$ and $1$ when $p=n$.
This gives the Hodge number.
\end{proof}

\bigbreak

\begin{thm}\label{mainthm:HodgehyperplaneGr}
Let $Y$ be a smooth hyperplane section of $Gr(k, n)$ where $n\geq 2k$. Then $H^*(Y)$ is of Hodge--Tate type if and only if
 either {\upshape (i) }
 $k\in \{1, 2\}$ or {\upshape (ii) } $k=3$ and $n\in \{6, 7, 8\}$ holds.
\end{thm}

\begin{proof}
By Proposition \ref{prop: nonhodge}, it remains to check the case when $k(n-k)\leq 2n$.
For $k=1,2$, the hyperplane sections are known to be Hodge--Tate.
Thus by Lemma \ref{lemma: nk}, it remains to  check the cases when $(k,n)$ belongs to $\{(3,6),\,(3,7),\,(3,8),\,(3,9),\,(4,8)\}$.

To compute the Hodge numbers, we consider the $\lambda$-class
$\lambda_y(Y)=\sum_{p\geq 0} y^p [\Omega_Y^p]\in K(Y)[y]$ in the polynomial ring of $y$ with coefficients in the $K$-theory  of $Y$ \cite{Hirzebruch}.
We have
$$(1+y\mathcal{O}(-Y))\cdot
\lambda_y(Y)=\lambda_y(X)|_Y.$$
So
$$\chi_y(Y):=
\sum_{p\geq 0} y^p \chi(Y,\Omega^p_Y)
=\chi(Y,\lambda_y(Y))=\chi\left(X,
\lambda_y(X)\frac{1-\mathcal{O}(-Y)}{1+y\mathcal{O}(-Y)}\right)\in \mathbb{Z}[y].$$
This can be obtained by using the Bott--Lefschetz  localization formula and taking the non-equivariant limit.
For $X=Gr(3,9)$, we have
\begin{align*}
\chi_y(Y)&
=-y^{17} + y^{16} - 2 y^{15} + 3 y^{14} - 4 y^{13} + 5 y^{12} - 7 y^{11} + 7 y^{10} - 6 y^{9} \\
&\qquad  + 6 y^{8} - 7 y^{7} + 7 y^{6} - 5 y^{5} + 4 y^{4} - 3 y^{3} + 2 y^{2} - y + 1.
\end{align*}
That is,
$$\begin{array}{|c|c|c|c|c|c|c|c|c|c|c|c|c|c|c|c|c|c|c|}\hline
p & 0 & 1 & 2 & 3 & 4 & 5 & 6 & 7 & 8 & 9 & 10 & 11 & 12 & 13 & 14 & 15 & 16 & 17\\\hline\hline
\chi(Y,\Omega_Y^p) &
1&-1&2&-3&4&-5&7&-7&6&-6&7&-7&5&-4&3&-2&1&-1\\\hline
\end{array}$$
By  Lefschetz theorem and Poincar\'e duality, we have
\begin{itemize}
    \item[(1)] $h^{i, j}(Y)= 0$ unless $i=j$ or $i+j=\dim Y=17$;
    \item[(2)]
    when $i+j<\dim Y=17$, $h^{i,j}(Y)=h^{i, j}(X)$.
\end{itemize}
Comparing them with the Poincar\'e polynomial of $Gr(3,9)$, we  conclude when $i+j=17$,
$$h^{8, 9}=h^{9,8}=2,\qquad h^{i, j}=0
\text{ otherwise}.$$
In particular,  $H^*(Y)$ is not of Hodge--Tate type.
By the same methods, we can work out the other cases. The following  diagrams list their Hodge diamonds.
$$\begin{array}{c@{\qquad}c@{\qquad}c}
\scriptstyle
Y\subset Gr(3,6)&
\scriptstyle
Y\subset Gr(3,7)&
\scriptstyle
Y\subset Gr(3,8)\\
\begin{subarray}{c}
    1\\1\\2\\3\\4\\3\\2\\1\\1
\end{subarray}\quad
\begin{subarray}{c}
    1\\1\\2\\3\\3\\3\\3\\2\\1\\1
\end{subarray}\quad&
\begin{subarray}{c}
    1\\1\\2\\3\\4\\4\\4\\4\\3\\2\\1\\1
\end{subarray}\quad
\begin{subarray}{c}
    1\\1\\2\\3\\4\\4\\5\\4\\4\\3\\2\\1\\1
\end{subarray}\quad&
\begin{subarray}{c}
    1\\1\\2\\3\\4\\5\\6\\7\\6\\5\\4\\3\\2\\1\\1
\end{subarray}\quad
\begin{subarray}{c}
    1\\1\\2\\3\\4\\5\\6\\6\\6\\6\\5\\4\\3\\2\\1\\1
\end{subarray}\quad
\end{array}
\qquad
\begin{array}{c@{\qquad}c}
\scriptstyle Y\subset Gr(3,9)&
\scriptstyle Y\subset Gr(4,8)\\
\begin{subarray}{c}1\\1\\2\\3\\4\\5\\7\\7\\8\\[-0.75ex]2\,\,\,\, 2\\[-0.75ex]8\\7\\7\\5\\4\\3\\2\\1\\1\end{subarray}\quad
\begin{subarray}{c}1\\1\\2\\3\\4\\5\\7\\7\\8\\8\\
8\\7\\7\\5\\4\\3\\2\\1\\1\end{subarray}\quad&
\begin{subarray}{c}
1\\1\\2\\3\\5\\5\\7\\7\\[-0.75ex]3\,\,\, 3\\[-0.75ex]7\\7\\5\\5\\3\\2\\1\\1
\end{subarray}\quad
\begin{subarray}{c}
1\\1\\2\\3\\5\\5\\7\\7\\8\\7\\7\\5\\5\\3\\2\\1\\1
\end{subarray}\quad
\end{array}$$
 Hence, for $k\geq 3$, $H^*(Y)$ is of Hodge--Tate type if and only if $k=3$ and $n\in \{6, 7, 8\}.$
\end{proof}

\begin{remark}
The method in the proof is well-known and works for any $Gr(k, n)$  in principle, while the computation of $\chi_y(Y)$ is not efficient in practice.
The polynomial $\chi_y(Y)$ can be efficiently computed via the Pieri rule of motivic Chern classes of Schubert cells over Grassmannian \cite{MCPieri}.
The readers can try it online: \url{https://cubicbear.github.io/PluckerHodge.html}.
\end{remark}

\subsection{Quantum Pieri rule for Hodge--Tate hyperplane sections}
 \subsubsection{Quantum Pieri rule for $X$} Here we review  some   facts on $QH^*(X)$ (see e.g. \cite{Buc}).

For $\lambda\in\mathcal{P}_{k, n}$, the Schubert subvariety $X_\lambda\subset X$ of codimension $|\lambda|$, associated to a fixed   complete flag $E_\bullet$ of $\mathbb{C}^n$, is defined by
$X_\lambda=\{V\in X\mid \dim V\cap E_{n-k+i-\lambda_i}\geq i, \quad i=1, \cdots, k\}$. The Schubert classes $\sigma_\lambda:=P.D.([X_\lambda])\in H^{2|\lambda|}(X, \mathbb{Z})$ form an additive basis of $H^*(X, \mathbb{Z})$. The dual basis is given by $\{(\sigma_\lambda)^\vee=\sigma_{\lambda^\vee}\}_\lambda$, where $\lambda^\vee=(n-k-\lambda_k, \cdots, \cdots, n-k-\lambda_1)$ is the dual partition. We simply denote the special partitions $p=(p, 0, \cdots, 0)$ and $1^p=(1, \cdots, 1, 0, \cdots, 0)$ where there are $p$ copies of 1. There is an exact sequence $$0\to \mathcal{S}\to \mathbb{C}^n\to \mathcal{Q}\to 0$$
of tautological vector bundles over $X$.  The fiber of the tautological subbundle $\mathcal{S}$ at a point $V\in X$ is   given by the vector space $V$.
The Schubert classes $\sigma_p$ (resp. $\sigma_{1^p}$) coincide with the $p$-th Chern classes $c_p(\mathcal{Q})$ (resp. $(-1)^pc_p(\mathcal{S})$). Hence, they are related by
$c(\mathcal{S})\cup c(\mathcal{Q})=1$, i.e.
\begin{align}\label{relationGr}
\sum_{i=0}^m(-1)^i\sigma_{1^i}\sigma_{m-i}=0
\end{align}
for all $m\geq 1$. Here   we take the convention $\sigma_{\lambda}=0$, whenever  $\lambda\not\in \mathcal{P}_{k, n}$.
There is a canonical ring isomorphism, where $\sigma_a$'s are polynomials in $\sigma_{1^b}$'s  read off from \eqref{relationGr},
\begin{align}\label{QHGr}
    QH^*(X)\cong \mathbb{C}[\sigma_1, \cdots,  \sigma_{1^k}, q_X]/(\sigma_{n-k+1}, \cdots, \sigma_{n-1}, \sigma_n+(-1)^k q_X).
\end{align}
 The quantum multiplications by   $\sigma_p$'s (or $\sigma_{1^p}$'s) are   known as the quantum Pieri rule, and were first provided by Bertram. We refer to \cite[Proposition 4.2]{BCFF} for the following form.
\begin{prop}[Quantum Pieri rule]
   Let $1\leq p\leq k$ and $\lambda\in \mathcal{P}_{k, n}$. In $QH^*(X)$, we have
   $$\sigma_{1^p}\star \sigma_\lambda=\sum_\mu\sigma_\mu+ q_X\sum_\nu \sigma_\nu, $$
the first sum over $\mu$  obtained by adding $p$ cells to $\lambda$  with no two in the same row,
   and the second sum over $\nu$ with $|\nu|=|\lambda|+p-n$ such that $\tilde \lambda_1-1\geq \tilde \nu_1\geq \tilde \lambda_2-1\geq \tilde \nu_2\geq \cdots\geq \tilde \lambda_{n-k}-1\geq \tilde \nu_{n-k}\geq 0$, which occurs only if $\lambda_1=n-k$ and where $\tilde \lambda\in \mathcal{P}_{n-k, n}$ denotes the transpose of $\lambda$.
\end{prop}

\subsubsection{Quantum Pieri rule for $Y$} Denote by $j: Y\hookrightarrow X$ the natural inclusion. It induces an algebra homomorphism $j^*: H^*(X)\to H^*(Y)$, with $j^*|_{H^{2i}(X)}$ an isomorphism of vector spaces for $0\leq 2i< \dim Y$.
Taking a line $\mathbb{P}^1\subset Y$, we have $H_2(Y, \mathbb{Z})=\mathbb{Z}[\mathbb{P}^1]$ and $H_2(X, \mathbb{Z})=\mathbb{Z}j_*[\mathbb{P}^1]$, and   simply denote by $d\in \mathbb{Z}$ a curve class under this identification.

 \begin{prop}[\protect{\cite[Proposition 5.13 and Theorem 5.11 (1)]{BePe}}] \label{prop: virfundamental} Assume $n\geq 2k>3$.
 \begin{enumerate}
     \item For $d\in \{1, 2\}$, $\overline{\mathcal{M}}_{0, n}(Y, d)$ is irreducible of expected dimension.
   \item
   For any $\alpha, \alpha'\in H^*(Y)$ and  $\gamma\in H^{2i}(X)$
      with $i<n-1$, we have $$\langle \alpha, \alpha', j^*\gamma\rangle_{1}^Y=\langle j_*\alpha, j_*\alpha', \gamma\rangle_{1}^X.$$

 \end{enumerate}
 \end{prop}
\begin{remark}
 Part (1) with $d=1$ is a consequence of the result in \cite{LaMa} as explained in \cite[Corollary 5.8]{BePe}. Then part (2) is obtained by a canonical argument with projection formula as in \cite[Lemma 5.12]{BePe}.
  Part (1) with $d=2$  was proved  in \cite[Theorem 5.11 (1)]{BePe} for  adjoint or quasi-minuscule Grassmannians (excluding type $G_2)$, and  also holds in our situation  by exactly the same arguments therein.
\end{remark}

Extend the morphism  $j^*: H^*(X)\to H^*(Y)$  to  a $\mathbb{C}$-linear map $j^*: QH^*(X)\to QH^*(Y)$  by defining $j^* (q_X^m)=q_Y^m$, which is distinct from the lifting  $j^*_q$ in Proposition \ref{prop: QLP}.

\begin{prop}\label{prop: qPieriY_deg1} Let $\lambda\in \mathcal{P}_{k, n}$,  $1\leq p\leq k$ and $\beta\in H_{\rm prim}(Y)$.  In $QH^*(Y)$, we have
      \begin{align*}
          j^*\sigma_{1^p}\star j^*\sigma_\lambda &\equiv j^*(\sigma_{1^p}\cup \sigma_\lambda)+j^*(\sigma_{1^p}\star (\sigma_\lambda\cup \sigma_1)-\sigma_{1^p}\cup (\sigma_\lambda\cup \sigma_1))\quad \mod q_Y^2,\\
           j^*\sigma_{1^p}\star \beta & \equiv 0\quad \mod q_Y^2.      \end{align*}
  \end{prop}
\begin{proof}
 Denote $m:=p+|\lambda|-(n-1)$.  Take an ordering  $\{\sigma_{\mu^{(1)}},\cdots, \sigma_{\mu^{(b)}}\}$ of the Schubert basis of $H^{2m}(X)$ such that   $H^{2m}(Y)$ has a basis   $\{j^*\sigma_{\mu^{(i)}}\}_{1\leq i\leq a}\bigcup\{\beta_i\}_{1\leq i\leq c}$
 for some $1\leq a\leq b$ and  primitive classes $\beta_i$'s which appear only if $\dim Y=2m$. Write
 $$j^*\sigma_{1^p}\star j^*\sigma_\lambda\equiv j^*\sigma_{1^p}\cup j^*\sigma_\lambda+ q_Y\sum_{i=1}^am_i j^*\sigma_{\mu^{(i)}}+q_Y\sum_{i=1}^c\tilde m_i \beta_i\mod q^2_Y.$$
Then for  any $1\leq i\leq c$, $\tilde m_i=\langle j^*\sigma_{1^p}, j^*\sigma_\lambda, \beta_i^\vee\rangle_1^Y= \langle \sigma_{1^p}, j_*(j^*\sigma_\lambda), j_*(\beta_i^\vee)\rangle_1^X=0$
by using   Proposition \ref{prop: virfundamental} and noting $j_*(\beta_i^\vee)=0$ (since $\beta_i^\vee$ is again a primitive class). For $1\leq i\leq a$,
$$m_i=\langle j^*\sigma_{1^p}, j^*\sigma_\lambda, (j^*\sigma_{\mu^{(i)}})^\vee\rangle_1^Y= \langle  \sigma_{1^p}, j_*(j^*\sigma_\lambda), j_*((j^*\sigma_{\mu^{(i)}})^\vee)\rangle_1^X.$$
 Since $Y$ is a hyperplane section of $X$, by projection formula we have
 $j_*(j^*\sigma_\lambda)=j_*(j^*\sigma_\lambda\cup 1)=\sigma_\lambda\cup j_*1=\sigma_\lambda\cup \sigma_1$.
 For any $\nu\in \mathcal{P}_{k, n}\setminus\{\mu^{(s)}\}_{a<s\leq b}$, we have
\begin{align*}\label{dualformula}
     \int_{[X]} j_*((j^*\sigma_{\mu^{(i)}})^\vee)\cup \sigma_\nu=\int_{[Y]} (j^*\sigma_{\mu^{(i)}})^\vee\cup j^*\sigma_\nu=\delta_{\mu^{(i)}, \nu}
 \end{align*}
For $a<s\leq b$, writing $j^*\sigma_{\mu^{(s)}}=\sum_{t=1}^a c_{st}j^*\sigma_{\mu^{(t)}}$,  we have
  $ \int_{[X]} j_*((j^*\sigma_{\mu^{(i)}})^\vee)\cup \sigma_{\mu^{(s)}}=\int_{[Y]} (j^*\sigma_{\mu^{(i)}})^\vee\cup (\sum_tc_{st}j^*\sigma_{\mu^{(t)}})= c_{si}$. It follows that
   $j_*((j^*\sigma_{\mu^{(i)}})^\vee)= \sigma_{\mu^{(i)}}^\vee+\sum_sc_{si}\sigma_{\mu^{(s)}}^\vee$.
   Hence,
   \begin{align*}
   q_Y\sum_{i=1}^am_i j^*\sigma_{\mu^{(i)}}
    &=q_Y\sum_{i=1}^{a} \big(\langle \sigma_{1^p}, \sigma_\lambda\cup \sigma_1, \sigma_{\mu^{(i)}}^\vee\rangle_1^X j^*\sigma_{{\mu}^{(i)}} + \sum_{s=a+1}^b\langle \sigma_{1^p}, \sigma_\lambda\cup \sigma_1, \sigma_{\mu^{(s)}}^\vee \rangle_1^Xc_{si} j^*\sigma_{{\mu}^{(i)}}\big)\\
     &=  q_Y\sum_{i=1}^{b} \langle \sigma_{1^p}, \sigma_\lambda\cup \sigma_1, \sigma_{\mu^{(i)}}^\vee\rangle_1^X j^*\sigma_{\mu^{(i)}}
     \\&=  j^*(\sigma_{1^p}\star (\sigma_\lambda\cup \sigma_1)-\sigma_{1^p}\cup (\sigma_\lambda\cup \sigma_1)).
        \end{align*}
  Here the last equality follows from the quantum Pieri rule for $X$, which shows that the quantum part of   $\sigma_{1^p}\star (\sigma_\lambda\cup \sigma_1)$ consists of the degree-one part of the quantum product.
\end{proof}

   The \textit{span} $Span(Z)$ of  a subvariety $Z\subset X$ is  the smallest vector subspace of $\mathbb{C}^n$ that contains all the $k$-dimensional spaces given by points of $Z$. It has     been used to study $QH^*(X)$ in \cite{Buc}.
 \begin{prop} \label{prop: deg2GW}
Let $k\geq 3$ and  $Z$  be a closed subvariety of $Y$ of dimension $m$. If  $n>2k+m$, then for any   $\alpha\in  H^*(Y)$, we have
   $$\langle P.D.[{\rm pt}], P.D.[Z],  \alpha\rangle_{2}^Y=0.$$
 \end{prop}
 \begin{proof}
     It suffices to show the case when $\alpha$ is represented by    a closed subvariety of dimension   $\tilde m = 2\dim Y-m-2(n-1)$. Assume
      $\langle [{\rm pt}], [Z],  \alpha\rangle_{2}^Y\neq 0$, then
      for any point $P\in Y$ and any closed subvariety $Z'$ of dimension $\tilde m$, there exists a conic $C$ passing through $P$, $Z$ and $Z'$ by Proposition   \ref{prop: virfundamental} (1). Say $\hat P \in C\cap Z$, then we have $\dim Span(P)=\dim(\hat P)=k$. Hence,
       $Span(P)\cap Span(\hat P)\neq 0$, by noting that they are both vector subspaces of $Span(C)$ while $\dim Span(C)\leq k+2$ by \cite[Lemma 1]{Buc}.

       Consider the configuration space
        $$B_Z:=\{(V_1, V_k, \bar V_k)\mid V_1\leq \bar V_k, \,\, V_1\leq V_k\in Z, \dim V_1=1, \dim V_k=\dim \bar V_k=k\}  $$
     as well as the natural projections $\pi_i$ by sending $(V_1, V_k, \bar V_k)$ to the $i$th vector space. Note that $B_Z$ is a closed variety, which is a fibration over $Z$ via $\pi_2$ with fiber at $V_k$ being a $Gr(k-1, n-1)$-bundle over $\mathbb{P}(V_k)$. Thus $\dim B_Z=\dim Z+ (k-1)+(k-1)(n-k)$. Hence,
      $$\mbox{codim}_X\pi_3(B_Z)=\dim X-\dim\pi_3(B_Z)\geq \dim X-\dim B_Z=
      n-2k-m+1.$$
     Since $Y$ is codimension $1$ in $X$,
     $\mbox{codim}_Y(Y\cap \pi_3(B_Z)) \geq n-2k-m>0$ by the hypothesis. Therefore, there exists $\bar V_k\in Y\setminus \pi_3(B_Z)$. Then for any $V_k\in Z$ and any 1-dimensional vector subspace $V_1$, either $V_1\not\leq \bar V_k$ or $V_1\not\leq  V_k$ holds. That is, $\bar V_k\cap V_k=0$, saying that the span of the point $\bar V_k$ in $Y$ and the span of the point $V_k\in Z\subset Y$ has zero intersection and resulting in a contradiction.
 \end{proof}

\begin{defn}
    We call $Y$ a Hodge--Tate hyperplane section, if $H^*(Y)$ is of Hodge--Tate type. 
\end{defn}

  \begin{prop}[Quantum Pieri rule for $Y$]\label{prop: quanPieriY} Let $Y$ be a Hodge--Tate hyperplane section. Take any $1\leq p\leq k$,  $\lambda\in \mathcal{P}_{k, n}$ and $\beta\in H_{\rm prim}(Y)$.  In $QH^*(Y)$, we have
      $$j^*\sigma_{1^p}\star j^*\sigma_\lambda= j^*(\sigma_{1^p}\cup \sigma_\lambda)+j^*(\sigma_{1^p}\star (\sigma_\lambda\cup \sigma_1)-\sigma_{1^p}\cup (\sigma_\lambda\cup \sigma_1));\qquad j^*\sigma_{1^p}\star \beta=0. $$
  \end{prop}
\begin{proof}
 Denote $q:=q_Y$. For $k \in \{1, 2\}$, we have $H_{\rm prim}(Y)=0$ and $p+|\lambda|\leq k+\dim Y=k+k(n-k)-1<2(n-1)=2\deg q$. Thus there   are no $q^d$-terms with $d\geq 2$ in   $j^*\sigma_{1^p}\star j^*\sigma_\lambda$.
 By Theorem \ref{mainthm:HodgehyperplaneGr}, it remains to consider the case when $k=3$ and $n\in \{6, 7, 8\}$.

 By degree counting again, there are no  $q^d$-terms with $d\geq 3$ (resp. $d\geq 2$) in   $j^*\sigma_{1^p}\star j^*\sigma_\lambda$ (resp. $j^*\sigma_{1^p}\star \beta$). Thus
by Proposition \ref{prop: qPieriY_deg1}, we have $j^*\sigma_{1^p}\star \beta=0$.

 For $p=1$, we further conclude that there are no  $q^2$-terms in   $j^*\sigma_{1}\star j^*\sigma_\lambda$. Indeed, when $n=6$, this holds by degree counting.
 When $n=7$,  a $q^2$-term occurs only if $12=2\deg q\leq 1+|\lambda|\leq 1+\dim Y=12$, implying $\lambda=(4,4,3)$.
 The coefficient of $q^2 \cdot 1$ in $j^*\sigma_1 \star j^*\sigma_{(4,4,3)}$ equals
 $\langle j^*\sigma_1, P.D.[{\rm pt}], P.D.[{\rm pt}]\rangle_2^Y$, and hence equals 0 by Proposition \ref{prop: deg2GW}.
When $n=8$, there are at most two possibilities by degree counting, being  the  coefficient of $q^2 \cdot 1$ (resp. $q^2j^*\sigma_1$) in $j^*\sigma_1 \star j^*\sigma_{(5,5,3)}$ (resp. $j^*\sigma_1 \star j^*\sigma_{(5,5,4)}$). They both equal
 $\langle j^*\sigma_1, P.D.[{\rm pt}], P.D.[\mathbb{P}^1]\rangle_2^Y$, and hence equal 0 by Proposition \ref{prop: deg2GW} again.
Hence, the first equality in the statement holds for $p=1$.

 For $p\in \{2, 3\}$,  it follows from  the associativity of quantum products that  $q^2$-terms do not occur.
 For instance for $p=2$ and $n=7$, $q^2$-terms occur at most in $j^*\sigma_{(1,1,0)}\star j^*\sigma_{(4,4,2)}$ and $j^*\sigma_{(1,1,0)}\star j^*\sigma_{(4,4,3)}$, by degree counting. By direct calculations, we have
 \begin{align*}
     j^*\sigma_1\star j^*\sigma_{(4,4,1)}&=j^*\sigma_{(4,4,2)}+qj^*\sigma_{(3,1,0)},\\
  j^*\sigma_{(1,1,0)}\star j^*\sigma_{(4,4,1)}\star j^*\sigma_1   &=q(j^*\sigma_{(3,2,0)}+j^*\sigma_{(4,1,0)})\star j^*\sigma_1=q(j^*\sigma_{(3,2,0)}+j^*\sigma_{(4,1,0)})\cup j^*\sigma_1+q\cdot q,\\
 qj^*\sigma_{(3,1,0)}\star j^*\sigma_{(1,1,0)} &=qj^*\sigma_{(3,1,0)}\cup j^*\sigma_{(1,1,0)}+q\cdot q.
 \end{align*}
Hence, $j^*\sigma_{(1,1,0)}\star j^*\sigma_{(4,4,2)}=q(j^*\sigma_{(3,2,0)}+j^*\sigma_{(4,1,0)})\cup j^*\sigma_1-qj^*\sigma_{(3,1,0)}\cup j^*\sigma_{(1,1,0)}$, which does not contain $q^2$-term.  The arguments for the remaining   few cases are similar.
\end{proof}

\subsubsection{Applications}

Denote by $p_{\alpha, V}(x)$ the characteristic polynomial  of the induced operator $\hat \alpha$ on an algebra $V$ containing $\alpha$. Using  the quantum Pieri rules, we have the following.
 \begin{lemma}
 \label{lemma: GrYchp}
 Let  $Y$ be a smooth hyperplane section of $X=Gr(3, n)$, where $n\in \{7, 8\}$.
 \begin{enumerate}
     \item For $n=7$,    $p_{\sigma_1^7, \mathcal{A}^0(X)}(x)=p_{(j^*\sigma_1)^6, \mathcal{A}^0(Y)}(x)=(128 - 13  x + x^2) (1 - 57   x - 289  x^2 + x^3).$
     \item For $n=8$, we    have   $p_{\sigma_1^8, \mathcal{A}^0(X)}(x)=(x - 1)^3 (1 - 1154   x + x^2) (6561   - 34   x + x^2)$
     and $ p_{\sigma_1^6\star \sigma_{1^2}, \mathcal{A}^0(X)}(x)
        = (1 - x) (1 + 478  x - x^2) (1 + x^2) (2187  + 6  x + x^2).$
     Moreover,
     $$ p_{(j^*\sigma_1)^7, \mathcal{A}^0(Y)}(x)=x^2 p_{\sigma_1^8, \mathcal{A}^0(X)}(x),\qquad p_{(j^*\sigma_1)^5\star j^*\sigma_{1^2}, \mathcal{A}^0(Y)}(x)=x^2 p_{\sigma_1^6\star \sigma_{1^2}, \mathcal{A}^0(X)}(x)$$
  \end{enumerate}
\end{lemma}

 \begin{proof}
     For $n=7$,  $\mathcal{A}^0(X)=\mathbb{C}\{1, \sigma_{(4,3,0)}, \sigma_{(4, 2, 1)}, \sigma_{(3, 3, 1)}, \sigma_{(3, 2, 2)}\}$
     and $\mathcal{A}^0(Y)=\mathbb{C}\{1, j^*\sigma_{(4,1,1)},$ $ j^*\sigma_{(3, 3, 0)}, j^*\sigma_{(3, 2, 1)}, j^*\sigma_{(2, 2, 2)}\}$ with
     $j^*\sigma_{(4,2,0)}= 2 j^*\sigma_{(4,1,1)}+ j^*\sigma_{(3, 3, 0)}- j^*\sigma_{(3, 2, 1)}+j^*\sigma_{(2,2,2)}$.

      For $n=8$,  $\mathcal{A}^0(X)=\mathbb{C}\{1, \sigma_{(5,3,0)}, \sigma_{(5, 2, 1)}, \sigma_{(4,4,0)}, \sigma_{(4, 3, 1)}, \sigma_{(4, 2, 2)}, \sigma_{(3,3,2)}\}$
     and $\mathcal{A}^0(Y)=\mathbb{C}\{1,$ $ j^*\sigma_{(5,2,0)},$ $ j^*\sigma_{(5, 1, 1)}, j^*\sigma_{(4, 2, 1)}, j^*\sigma_{(4, 3, 0)},j^*\sigma_{(3, 2, 2)}, j^*\sigma_{(3, 3, 1)}, j^*\sigma_{(5, 5, 4)}, \beta\}$, where $\beta$ is a primitive class and we have $j^*\sigma_1\star \beta=0$.

     Then the statements follow from direct  calculations by using  Proposition \ref{prop: quanPieriY}.
\end{proof}

 \begin{prop}\label{prop:conjpart2true}
    Let   $Y$ be a smooth hyperplane section of $X=Gr(3, n)$ with $n\in\{7, 8\}$. There is an isomorphism of algebras $$\mathcal{A}^0(X)\overset{\cong}{\longrightarrow} \mathcal{A}^0_\perp(Y)$$ with $\sigma_1^{r_X}\mapsto ((j^*\sigma_1)^{r_Y})_\perp$.
 \end{prop}
 \begin{proof}
     Denote $\alpha_X:=\sigma_1^{r_X}$ and $\alpha_Y:=(j^*\sigma_1)^{r_Y}$. For $n=7$, by direct calculations,    $\{1, \alpha_X, \cdots, \alpha_X^4\}$  (resp.  $\{1, \alpha_Y, \cdots, \alpha_Y^4\}$) form an additive basis of $\mathcal{A}^0(X)$ (resp. $\mathcal{A}^0(Y)$). Therefore the minimal polynomial of $\alpha_X$ (resp. $\alpha_Y$) is given by the characteristic polynomial of it. Therefore, by Lemma \ref{lemma: GrYchp},  $\mathcal{A}^0(X)$ (resp.  $\mathcal{A}^0(Y)$) is isomorphic to $\mathbb{C}[x]/(p_{\alpha_X, \mathcal{A}^0(X)}(x))$ by sending $\alpha_X$ (resp. $\alpha_Y$) to $x$.

  Now we consider $n=8$ and denote $\tilde \alpha_X:=\sigma_1^{r_X-2}\star \sigma_{1^2}, \tilde \alpha_Y:=  (j^*\sigma_1)^{r_Y-2}\star j^*\sigma_{1^2} $. By exactly the same argument for $n=7$,
   we conclude that $\mathcal{A}^0(X)$ (resp.  $\mathcal{A}^0_\perp(Y)$) is isomorphic to $\mathbb{C}[x]/(p_{\tilde\alpha_X, \mathcal{A}^0(X)}(x))$ by sending $\tilde \alpha_X$ (resp. $(\tilde\alpha_Y)_\perp$) to $x$. By direct calculations, we see that   $\alpha_X$ and $(\alpha_Y)_\perp$ have exactly the same linear expansion in terms of  $\{\tilde \alpha_X^i\}_i$ and
   $\{(\tilde \alpha_Y^i)_\perp\}_i$ respectively. Hence $\alpha_X$ is sent to $(\alpha_Y)_\perp$ under the isomorphism.
 \end{proof}

\bigbreak

\begin{thm}\label{thm: conjtrueforGr}
 Let   $Y$ be a smooth hyperplane section of $X=Gr(3, n)$ with $n\in\{7, 8\}$. Then both Proposition \ref{prop: QLP} and  Conjecture \ref{conj: QLP} hold  for $Y$.
\end{thm}

 \begin{proof}
     Conjecture \ref{conj: QLP}  and  Part (2)   of Proposition \ref{prop: QLP} follow directly  from Proposition \ref{prop:conjpart2true} and Lemma \ref{lemma: GrYchp} respectively.

    Recall   $QH^*(Gr(3, n))\cong \mathbb{C}[\sigma_1, \sigma_{1^2}, \sigma_{1^3}, q_X]/(\sigma_{n-2}, \sigma_{n-1}, \sigma_n-q_X)$,  where $\sigma_{j}$ can be obtained by the  relation \eqref{relationGr}.
Denote by $e_1=j^*\sigma_1, e_2=j^*\sigma_{1^2}$ and $e_3=j^*\sigma_{1^3}$. Note $j^*_q(\sigma_{n-2})=0$, as $n-2<r_Y$.
To verify part (1) of Proposition \ref{prop: QLP}, it suffices to show that $$j^*_q(\sigma_{n-1})=0,\qquad j^*_q(\sigma_n)=j^*_q(q_X)=q_Y e_1,$$ where  $j^*_q(\sigma_{n-1})=h_6$, $j^*_q(\sigma_{n})=h_7$ for $n=7$ (resp. $j^*_q(\sigma_{n-1})=h_7$, $j^*_q(\sigma_{n})=h_8$ for $n=8$) with
    \begin{align*}
h_6&=e_{1}^{6} - 5 e_{1}^{4} e_{2} + 6 e_{1}^{2} e_{2}^{2} + 4 e_{1}^{3} e_{3} - e_{2}^{3} - 6 e_{1} e_{2} e_{3} + e_{3}^{2},\\
h_7&=e_{1}^{7} - 6 e_{1}^{5} e_{2} + 10 e_{1}^{3} e_{2}^{2} + 5 e_{1}^{4} e_{3} - 4 e_{1} e_{2}^{3} - 12 e_{1}^{2} e_{2} e_{3} + 3 e_{2}^{2} e_{3} + 3 e_{1} e_{3}^{2},\\
h_8&=e_{1}^{8} - 7 e_{1}^{6} e_{2} + 15 e_{1}^{4} e_{2}^{2} + 6 e_{1}^{5} e_{3} - 10 e_{1}^{2} e_{2}^{3} - 20 e_{1}^{3} e_{2} e_{3} + e_{2}^{4} + 12 e_{1} e_{2}^{2} e_{3} + 6 e_{1}^{2} e_{3}^{2} - 3 e_{2} e_{3}^{2}.
\end{align*}
Then we are done by direct calculations using Proposition \ref{prop: quanPieriY}.
\end{proof}

\subsection{Semisimplicity of $QH^*(Y)$}
Here we investigate the semisimplicity of $QH^*(Y)$ for a smooth hyperplane section $Y$ of $X=Gr(k, n)$. We will need to apply  the following proposition in the case $(k, n)=(3, 8)$.

{
Let $\pi: \mathcal{Y}\to U$ be a projective smooth morphism.
Let $p\in U$ be any point and denote $\mathcal{Y}_p:=\pi^{-1}(p)$ the fiber of $\pi$ at $p$.
There is a \emph{monodromy action} (see e.g.  \cite[Chapter 3]{Voisin}) of the fundamental group $\pi_1(U,p)$ on the classical cohomology $H^*(\mathcal{Y}_p)$.

\begin{prop}[\protect{\cite[Theorem 4.3]{LiTi}, see also \cite[Corollary 3.2]{Hu15}}] \label{prop: monodromy}
The monodromy action of $\pi_1(U,p)$ on $H^*(\mathcal{Y}_p)$ naturally extends to an action on $QH^*(\mathcal{Y}_p)$ and preserves the quantum product.
\end{prop}

For $Y$ in $Gr(3,  8)$, we consider the family
$\pi|_{U}: \mathcal{Y}^\circ\to U$ obtained by restricting the natural projection
$$\pi: \mathcal{Y}=\{(V,\tau)\in Gr(3,8)\times \Lambda^3(\mathbb{C}^8)^*\mid
\tau|_V =0\}\longrightarrow \Lambda^3(\mathbb{C}^8)^*.$$
to the unique Zariski open $GL_8$-orbit   $U$ of $\Lambda^3(\mathbb{C}^8)^*$. More precisely,  the first projection
$\mathcal{Y}\to Gr(3, 8)$ is  a vector bundle over $Gr(3,8)$ whose fiber at $V$ is the space $\ker [\Lambda^3(\mathbb{C}^8)^*\to \Lambda^3V^*]$. Therefore, the smooth hyperplane section  $Y$ of $Gr(3, 8)$ can be realized as a fiber over any $p\in U$.    Notice $H_{\rm prim}(Y)=\mathbb{C}\{\beta\}$.

\begin{lemma}\label{lemma: monobeta} For $Y$ in $Gr(3, 8)$,
there exists  $\sigma\in \pi_1(U)$ such that
$\sigma \beta=-\beta$ under  monodromy action.
\end{lemma}
\begin{proof}
By Deligne invariant cycle theorem \cite[Theorem 4.24]{Voisin}, the image of the restriction
$$H^*(X,\mathbb{Q})\cong H^{*}(\mathcal{Y},\mathbb{Q})\to H^{*}(Y,\mathbb{Q})$$
is the monodromy invariant subalgebra. In particular, the restriction of $\sigma_1\in H^2(X,\mathbb{Q})$ is invariant, so the primitive space $\mathbb{Q}\beta$ is preserved. Since $\beta$ is not in the image, there exists    $\sigma\in\pi_1(U)$ such that $\beta \neq \sigma\beta\in \mathbb{Q}\beta$.
Note that $\int_{Y}\beta^2=\int_Y \sigma\beta^2\neq 0$, the only possibility is $\sigma\beta=-\beta$.
\end{proof}

\begin{thm}\label{mainthm:QHhyperplaneGr}
Let $Y$ be a smooth hyperplane section of $Gr(k, n)$ where $n\geq 2k$. Then $QH^*(Y)$ is generically semisimple if and only if one of the following holds:
$$(a)\,\, k=1; \qquad (b)\,\, k=2 \mbox{ and either } n=4 \mbox{ or } n \mbox{ is odd}; \qquad (c)\,\, k=3  \mbox{ and } n \in \{7, 8\}.$$
\end{thm}

\begin{proof}
   Let $X=Gr(k, n)$ with $n\geq 2k$. For $k=1$, $Y$ is a projective space and hence $QH^*(Y)$ is generically semisimple.
 For $X=Gr(2, 4)$, $X$ is a quadric in $\mathbb{P}^5$, and hence $Y$ is a quadric in $\mathbb{P}^4$ with generically semisimple quantum cohomology. For $k=2$ and    $n=2m$ with $m>2$,  $Y\cong SG(2, 2m)$ and  $QH^*(Y)$ is known to be non-semisimple \cite{ChPe}.
 For    $k=2$ and  $n=2m+1$ with $m>1$, $Y\cong SG(2, 2m+1)$ is a quasi-homogeneous variety with  generically semisimple quantum cohomology \cite{Pec, Per}. For $k>3$ or ($k=3$ and $n>8$),
 $QH^*(Y)$ is not semisimple by Theorem \ref{mainthm:HodgehyperplaneGr} and Proposition \ref{prop:hodge}.

 Now we assume $k=3$. It remains to investigate the cases $n\in \{6, 7, 8\}$. For $n=6$, we notice that $\dim \mathcal{A}^{\bar 1}=b_2(Y)+b_{12}(Y)=1+2\neq 4=b_{8}(Y)=\dim \mathcal{A}^{-\bar 1}$. Hence, $QH^*(Y)$ is not semisimple in this case.

 Now we consider $n\in \{7, 8\}$. Notice $\mathcal{A}^0_\perp (Y)=\mathcal{A}^0(Y) $ and $Rad(Y)=0$ if $n=7$.  It is well known that $QH^*(X)$ is semisimple, so is the subalgebra $\mathcal{A}^0(X)$. By Proposition \ref{prop:conjpart2true},
  $\mathcal{A}^0_\perp (Y) \cong \mathcal{A}^0(X)$ is semisimple. By Lemma \ref{lemma: GrYchp},
  the restriction of  $(j^*\sigma_1)^{r_Y}$ to the complement of $Rad(Y)$ in $QH^*(Y)$ is invertible, so is $j^*\sigma_1$. Hence,
   $\mathcal{A}^0_\perp(Y)[j^*\sigma_1]\cong \mathcal{A}^0_\perp(Y)[y]/(y^{r_Y}-(j^*\sigma_1)^{r_Y})$ is semisimple by Lemma \ref{lemma: AtoAy}. This settles the case $n=7$. Now we consider $n=8$, for which   $\dim \mathcal{A}^0_\perp(Y)[j^*\sigma_1]=49$.
   We have $QH^*(Y)=Rad(Y)\bigoplus \mathcal{A}^0_\perp(Y)[j^*\sigma_1]$, where $Rad(Y)=\mathbb{C}\{\beta, \gamma\}$ with
    $$\gamma:={2} \sigma_{(5,5,4)} - \sigma_{(3,2,2)} +
 \sigma_{(3,3,1)} + \sigma_{(4,2,1)} - 2 \sigma_{(4,3,0)} - 3\sigma_{(5,1,1)}+ 2\sigma_{(5,2,0)}\in H^*(Gr(3, 8)).$$
 Then we are done by showing that $Rad(Y)$ is semisimple. Indeed, 
  since $Rad(Y)$ is a two-dimensional unital algebra factor,
the nilpotent elements of $Rad(Y)$ form a subspace of dimension at most $1$ and any nilpotent element has nilpotent index no more than $2$.
Thus
  $j^*\gamma$ is not nilpotent, by noting $\int_Y(j^*\gamma)\cup (j^*\gamma)=6\neq 0$. 
If $a\beta+bj^*\gamma$ is nilpotent,
then so is $-a\beta+bj^*\gamma$ by applying the monodromy action by Proposition \ref{prop: monodromy} and Lemma \ref{lemma: monobeta}.
This forces $b=0$. Then it follows from
 $\int_Y \beta^2\neq 0$ that  $a=0$.
\end{proof}

\subsection{Proof of Lemma  \ref{lem:comblemm}}\label{section:lemmaproof}
 We discuss all the possible $n-i$.

\subsection*{Case I: $n-i\in \{1, 2\}$.}
If $n-i=1$, only the empty partition is a $1$-core partition, but our assumption implies $k(n-k)\geq k^2>k\geq i$. Thus it is impossible.

If $n-i=2$, then the $2$-core partition must be a staircase. The maximal $2$-core partition in  $\mathcal{P}_{k, n}$ is $(k, \cdots, 2, 1)$  and has $\frac{k(k+1)}{2}$ cells.
By the  assumption, we have
 $$\frac{k(k+1)}{2}\geq k(n-k)-i.$$
Hence,   $\frac{k(k+1)}{2}>k(n-k)-n$, implying
$\frac{k(k+1)}{2}+k^2>(k-1)n\geq 2(k-1)k$ and consequently   $k<5$. Combining the above inequality with $3\leq k\leq {n\over 2}$, we obtain  $(k,n,i)\in\{(3,6,4),(4,8,6)\}$, see the first two diagrams in Example \ref{ex:kni}.
\bigbreak

To proceed in the remaining cases, we use the  bijection $\mathcal{P}_{k, n}\to {[n]\choose k}$ defined by $\lambda\mapsto A=\{a_1,\ldots,a_k\}$ with   $a_i=\lambda_i+k-i+1$ for all $i$. Under this identification, we have
\begin{enumerate}[label=(\roman*)]
    \item\label{cond1}
  $|\lambda|\geq k(n-k)-i \Longleftrightarrow \sum_{i=1}^ka_i \geq n+(n-1)+\cdots+(n-k+1)-i$;
    \item\label{cond2}
  $\lambda$ is $(n-i)$-core $\Longleftrightarrow$
    if $a\in A$ and $ a-(n-i)\geq 1$ then $a-(n-i)\in A$.
\end{enumerate}
Here the equivalence in \ref{cond2} follows from the fact that the multiset of hook lengths of $\lambda$ is given by $
 \bigcup_{i=1}^k
\left(
\{1,\ldots,a_i-1\}
\setminus \bigcup_{j>i}\{a_i-a_j\}
\right)$, see \cite[proof of Lemma 7.21.1]{St99}.

\subsection*{Case II: $n-i\geq 3$ and   $i<n-i$}
Let us consider
$$B=\{a\in A\mid n-i<a\leq n\}.$$
Since $n-i<n<2(n-i)$, {we  have the disjoint union }
$$A=B\cup B'\cup C\qquad  $$
where $B'=\{b-(n-i)\mid b\in B\}$ and $C=A\setminus (B\cup B')$.
Assume $|B|=b$. Note that
$$k=|B|+|B'|+|C|\geq 2b,\qquad k-b=|B'|+|C|\leq n-i$$
and
$$
\texttt{\#}\big(A\cap \{1,\ldots,i\}\big)\geq b,\qquad
\texttt{\#}\big(A\cap \{1,\ldots,n-i\}\big)\geq
k-b.
$$
That is,
$$\begin{array}{c}
a_b<\cdots<a_2<a_1\leq n\\
a_{k-b}<\cdots<a_{b+2}<a_{b+1}\leq n-i\\
a_{k}<\cdots<a_{k-b+2}<a_{k-b+1}\leq i.
\end{array}$$
This implies
\begin{align*}
a_1+\cdots+a_k
&\leq n+(n-1)+\cdots+(n-b+1)\\
&\qquad + (n-i)+(n-i-1)+\cdots+(n-i+1-k+2b)\\
&\qquad \qquad +i+(i-1)+\cdots+(i+1-b)\\
& =n+(n-1)+\cdots+(n-k+1)-(k-2b)(i-b)-b(n-k+b-i).
\end{align*}
By the assumption \ref{cond1}, we have
\begin{equation}\label{eq:eqfrom(i)1}
0\geq (k-2b)(i-b)+b(n-k+b-i)-i,
\end{equation}
implying $b>0$ (otherwise we would have $0\geq ki-i>0$).
If $n-k+b-i\leq 1$, then we have
$k> k-b\geq n-i-1>\tfrac{n}{2}-1$,
contradicting   $k\leq \tfrac{n}{2}$.
Thus we  can assume $n-k+b-i\geq 2$.

If $k-2b>0$, then it follows from 
$\eqref{eq:eqfrom(i)1}$ that $0\geq (i-b)+2b-i=b>0$,
a contradiction.
It follows that $k=2b$.
Since $k\geq 3$, we have $b\geq 2$ and $n\geq 4b$.
Using $i<n/2$, we obtain
\[
\begin{aligned}
b(n-b)-(b+1)i
&> b(n-b)-\frac{(b+1)n}{2} 
 \geq 2b(b-1)-b^2
 = b(b-2)\geq 0.
\end{aligned}
\]
This contradicts the inequality   \eqref{eq:eqfrom(i)1}, where $k=2b$ in this case.  

\subsection*{Case III:  $n-i\geq 3$ and   $i\geq n-i$}
Let us consider
$$M=\{a\in A\mid i<a\leq n\},\qquad m=|M|.$$
By \ref{cond2}, we have an injective map
$\phi:M\to \{1,\ldots,n-i\}\cap A$
such that $\phi(a)= a-\ell(n-i)$ for some $\ell\geq 0$.
Moreover since $i\geq n-i$,
we have $a-(n-i)\geq a-i>0$,
hence $\ell\geq 1$ and in particular $a\neq \phi(a)$.
Now $M\cap \phi(M)=\varnothing$,  we can conclude that
$$k\geq 2m,\qquad \texttt{\#}(A\cap \{1,\ldots,n-i\})\geq m.$$

We first exclude $m=0$. If $M=\varnothing$, then $a_1\leq i$ and we have $a_k\leq n-i$ (
otherwise, condition \ref{cond2} would imply $a_k-(n-i)\in A$, contradicting the minimality of $a_k$).
Since the elements of $A$ are distinct, we obtain
\[
\begin{aligned}
\sum_{j=1}^k a_j
&\leq (k-1)i-\frac{(k-1)(k-2)}{2}+(n-i) 
 < \sum_{j=1}^k(n-j+1)-i,
\end{aligned}
\]
where the second inequality follows from $k\geq3$ and
$n-i\geq3$. This contradicts condition \ref{cond1}.
Hence $m\geq1$.
\subsection*{Subcase III.1}
When $|M|=1$, i.e. $M=\{a_1\}$,
we have
$$a_k\leq n-i,\qquad a_{k-1}<\cdots <a_3<a_2\leq i,\qquad a_1\leq n.$$
Then
\begin{align*}
a_1+\cdots+a_k
& \leq n+i+(i-1)+\cdots+(i-k+3)+(n-i)\\
& = n+(n-1)+(n-2)+\cdots+(n-k+1)-(k-2)(n-i-1)-(i-k+1).
\end{align*}
By assumption \ref{cond1} and noting $n-i\geq 3$, we have
$$0\geq (k-2)(n-i-1)-k+1\geq 2(k-2)-k+1=k-3\geq 0.$$
So $k=3$, and all equalities should be achieved at each step, i.e.
$$n-i=3,\qquad a_3=n-i,\qquad a_2=i,\qquad a_1=n.$$
Then $a_3=3$.
Since $a_2>a_3=n-i$, by \ref{cond2}, this forces $a_2-(n-i)=a_3$, i.e. $i=6$ and $n=9$. So the only case is  $(k,n,i)=(3,9,6)$, see the third diagram in Example \ref{ex:kni}.

\subsection*{Subcase III.2}
When $|M|=m\geq 2$, we have
$$a_{k}<a_{k-1}<\cdots<a_{k-m+1}\leq n-i\qquad\mbox{and}\qquad
a_{k-m}<\cdots <a_{m+2}<a_{m+1}\leq i.$$
Then
\begin{align*}
a_1+\cdots+a_k
& \leq n+(n-1)+\cdots+(n-m+1)\\
&\qquad  +i+(i-1)+\cdots+(i-(k-2m)+1)\\
&\qquad \qquad +(n-i)+(n-i-1)+\cdots+(n-i-m+1)\\
& = n+(n-1)+\cdots+(n-k+1)
  -(k-2m)(n-m-i)-m(i+m-k).
\end{align*}
By \ref{cond1}, we have
\begin{equation}\label{eq:eqfrom(i)2}
0\geq (k-2m)(n-m-i)+m(i+m-k)-i
\end{equation}
If $n-m-i=0$, then it follows from \eqref{eq:eqfrom(i)2} that
$$
0=m(n-k)-i\geq 2(n-k)-i\geq n-i>0,$$
 a contradiction.
So $n-m-i>0$, and by \eqref{eq:eqfrom(i)2} we have
$$0
\geq (k-2m)+2(i+m-k)-i=i-k\geq n/2-n/2\geq 0.$$
 
This shows $i=k$ and all equalities should be achieved at each step.
Then $i=k=n/2$,
$m=2$ (since $i+m-k=m>0$) and
$k-2m=0$ (since if $n-m-i=i-2=1$,
we can solve $(k,n,i)=(3,6,3)$ contradicting  $k\geq 2m$).
So the only possibility is $(k,n,i)=(4,8,4)$ with $a_1=n=8$, $a_2=a_1-1=7$, $a_3=n-i=4$ and $a_4=a_3-1=3$, see the last diagram in Example \ref{ex:kni}.
Now the proof of Lemma \ref{lem:comblemm} is completed.

\section{Smooth hyperplane sections of generalized Grassmannians}
In this section, we consider a generalized Grassmannian $X=G/P_k$. There is an embedding $X\hookrightarrow \mathbb{P}(V_{\omega_k})$, where $\omega_k$ is the $k$-th fundamental weight and  $V_{\omega_k}$ denotes the corresponding fundamental representation of $G$.
Let $Y$ be a smooth hyperplane section of $X$, given by the intersection of a general hyperplane in $\mathbb{P}(V_{\omega_k})$ with $X$.
We begin with some
(co)adjoint Grassmannians   in Table \ref{tab: coadj}.

\subsection{Hyperplane sections of $C_n/P_2$} The coadjoint  Grassmannian $G/P_2$ of type $C_n$ is the symplectic Grassmannian $SG(2, 2n)=\{V\leq \mathbb{C}^{2n}\mid \dim V=2,\, \omega(V, V)=0\}$, where   $\omega$ is a symplectic form on $\mathbb{C}^{2n}$ and $n\geq 3$.  It can be realized as a smooth hyperplane section of $Gr(2, 2n)$. In particular,   $r_X=2n-1$ and $\dim X=4n-5$.  
The following theorem verifies  \cite[Conjecture 1.10 (2)]{BePe}  for type $C$, and provides a new proof of the non-semisimplicity of $QH^*(SG(2, 2n))$ as well.

\begin{thm}\label{thm: type C} Let $n\geq 3$.
    For a smooth hyperplane section $Y$ of $X=SG(2, 2n)$, both $QH^*(X)$ and $QH^*(Y)$ are non-semisimple.
\end{thm}
\begin{proof}
Note that the odd Betti numbers $b_{2j+1}(X)=\dim H^{2j+1}(X)$ all vanish.
Since the two-step flag variety $SFl(1,2; 2n)$ of type $C$ is a $\mathbb{P}^1$-bundle over $SG(2,2n)$, and also a $\mathbb{P}^{2n-3}$-bundle over $\mathbb{P}^{2n-1}$, we have the  Poincar\'e polynomial (with respect to complex degree)
\begin{align*}
p_X(t)
& =
\frac{1-t^{2n}}{1-t}
\frac{1-t^{2n-2}}{1-t}\bigg/(1+t)\\
& = \frac{(1-t^{2n})(1-t^{2n-2})}{(1-t)(1-t^2)}
 = (1+t+\cdots +t^{2n-1})
(1+t^2+\cdots +t^{2n-4}).
\end{align*}
 
Thus for $0\leq i\leq 2n-3$, we have  $b_{2i}(X)
=\begin{cases}
\frac{i+2}{2}, & \mbox{if } i \text{ is even},\\
\frac{i+1}{2}, & \mbox{if } i \text{ is odd}.
\end{cases}$

Since $r_X=2n-1$ and $\dim X=4n-5$, we have $\dim \mathcal{A}^{-\bar 2}(X)=b_{2(2n-3)}(X)={2n-3+1\over 2}=n-1$ and $\dim \mathcal{A}^{\bar 2}(X)=b_{2\cdot 2}(X)+b_{2(2n+1)}(X)=b_{2\cdot 2}(X)+b_{2(2n-6)}(X)={2+2\over 2}+{2n-6+2\over 2}=n\neq \dim \mathcal{A}^{-\bar 2}$. Hence, $QH^*(X)$ is not semisimple by Theorem \ref{mainthm:ss} (2).

 Note $r_Y=2n-2$ and $\dim Y=4n-6$. By \cite[Theorem 1.2 and Lemma 3.1]{BePe}, $Y$ is invariant under the natural action of a maximal torus of $G$, and its torus-fixed loci coincide with the torus-fixed loci of $X$. Therefore $X$ and $Y$ have the same Euler number.  It follows that
  $b_{2(2n-3)}(Y)=b_{2(2n-3)}(X)+b_{2(2n-2)}(X)=2n-2$ and $b_{2i}(Y)=b_{2i}(X)$ for $0\leq i\leq 2n-4$. Thus
  $\dim \mathcal{A}^{\bar 1}(Y)=b_{2}(Y)+b_{2(2n-1)}(Y)=b_{2}(Y)+b_{2(2n-5)}(Y)=1+{2n-5+1\over 2}=n-1\neq b_{2(2n-3)}(Y)=\dim \mathcal{A}^{-\bar 1}$. Hence, $QH^*(Y)$ is not semisimple by Theorem \ref{mainthm:ss} (2).
\end{proof}
\begin{remark}
   The same argument  also works for $Y$ in 
$F_4/P_4$.
\end{remark}

\subsection{Hyperplane sections of $D_n/P_2$}
For $n\geq 4$, the (co)adjoint variety $D_n/P_2$  is the Grassmannian of isotropic planes in $\mathbb{C}^{2n}$ with respect to a non-degenerate symmetric bilinear form.
The  argument for the type $C$ case  does not  work for the type $D$ case.
Here we use the monodromy technique, which works not only in the type $D$ case.

Let  $G/P$ denote a coadjoint Grassmannian of type $D$ or $E$.
Consider  the  natural projection
$$
\pi:\mathcal{Y}:=\{(gP,v)\in G/P\times \mathfrak{g}\mid v\in gL\}
\longrightarrow \mathfrak{g}$$
where $L\subset \mathfrak{g}$ is the span of root vectors and Cartan subalgebra except the lowest root space.
Denote by $\kappa$ the Killing form of $\mathfrak{g}$.
Then the fiber $\mathcal{Y}_x$ at $x\in \mathfrak{g}$ can be identified with the hyperplane section defined by the linear equation $\kappa(x,-)$ over $\mathfrak{g}$ under the Pl\"ucker embedding $G/P\hookrightarrow \mathbb{P}(\mathfrak{g})$.
Moreover, for any $x$ in the set $U=\mathfrak{g}^{rs}$ of regular semisimple elements of $\mathfrak{g}$, the fiber $\mathcal{Y}_x$ is smooth, being a smooth hyperplane section $Y$ of $X$.
It is well-known that
the fundamental group of $\mathfrak{g}^{rs}$ is
the braid group of the Weyl group $W$ of $G$.

\begin{lemma}
The monodromy action factors through the Weyl group $W$, and as a $W$-representation,
$$ H^*(Y)= H^*(Y)_{\rm inv}\oplus H^*(Y)_{\rm std}$$
with $H^*(Y)_{\rm std}\cong \chi_{\rm std}$
the standard (i.e. reflection) representation of $W$.
\end{lemma}
\begin{proof}
We will use Springer correspondence in terms of perverse sheaves and Fourier transformation (see e.g. \cite{Gi98}).
Let $\mathbb{C}_Z$ be the constant sheaf of a variety $Z$.
Denote by $IC(Z,\chi)$ the intersection complex of the local system $\chi$ over $Z$ and we omit $\chi$ if $\chi=\mathbb{C}_Z$.
By abuse of notation, throughout this proof we write $f_*$ for $Rf_*$, suppress cohomological
shifts, and regard multiplicity spaces as ungraded, with multiplicities summed
over all shifts.
By the Decomposition Theorem \cite{BBD}, we have
\begin{equation}\label{eq:decofpiH}
\pi_*\mathbb{C}_{\mathcal{Y}} = \bigoplus_{\chi}IC(\mathfrak{g}^{rs},\chi)\otimes V_\chi\oplus
(\text{perverse sheaves with smaller support})
\end{equation}
where the sum goes over irreducible representations of $\pi_1(\mathfrak{g}^{rs})$, and $V_{\chi}$ is the multiplicity space of $\chi$ in $H^*(Y)$.
Let us consider the analogue $\rho$ of Springer resolutions
$$
\rho:\mathcal{Y}^\perp:=\{(gP,v)\in G/P\times \mathfrak{g}\mid v\in \operatorname{span}(ge_\theta)\}
\longrightarrow \mathfrak{g}$$
where $\theta$ is the highest root and $e_\theta$ is the corresponding root vector.
The image of $\rho$ contains only two nilpotent orbits,
the zero orbit $\mathbb{O}_0=\{0\}$ and the minimal $G$-orbit $\mathbb{O}_{\min}=Ge_\theta$.
Since $\rho$ is one-to-one over $\mathbb{O}_{\min}$, by the Decomposition Theorem \cite{BBD} again, we have
$$\rho_*\mathbb{C}_{\mathcal{Y}^\perp}
=IC(\mathbb{O}_{\min})\oplus
IC(\mathbb{O}_0)\otimes V
$$
for some coefficient  space $V$.
Let us denote by $\mathscr{F}$ the Fourier transformation between $\mathfrak{g}$ and $\mathfrak{g}^*\simeq \mathfrak{g}$, where the isomorphism is given by the Killing form.
Since $\mathfrak{g}$ is simply-laced,
by the Springer correspondence, we have
$$
\mathscr{F}\big(IC(\mathbb{O}_0)\big) = IC(\mathfrak{g}^{rs}),\qquad
\mathscr{F}\big(IC(\mathbb{O}_{\min})\big) = IC(\mathfrak{g}^{rs},\chi_{\rm std}),
$$
see for example \cite{Jut}.
Since Fourier transformation is functorial, similarly to \cite[Claim 8.4]{Gi98}, we have
$\mathscr{F}\big(\rho_*\mathbb{C}_{\mathcal{Y}^\perp}\big)
=\pi_*\mathbb{C}_{\mathcal{Y}}$.
This implies
$$\pi_*\mathbb{C}_{\mathcal{Y}}
=IC(\mathfrak{g}^{rs},\chi_{\rm std})\oplus
IC(\mathfrak{g}^{rs})\otimes V.$$
Comparing with \eqref{eq:decofpiH}, the multiplicity space $V_\chi=0$ unless $\chi=\chi_{\rm std}$ or $\chi_{\rm inv}$ and $\dim V_{\chi}=1$ for $\chi=\chi_{\rm std}$.
Hence, the decomposition in the statement follows.
\end{proof}

\begin{lemma}\label{len:Symstd}
For any finite irreducible Coxeter group $W$ not of type $A$, we have
$$\dim (\operatorname{Sym}^2 \chi_{\rm std})^{W}=1,\qquad
\dim (\operatorname{Sym}^3 \chi_{\rm std})^{W}=0.$$
\end{lemma}

\begin{proof}
By a theorem of Chevalley, for example \cite[Section 3.5]{Hum}, we have an algebra isomorphism to the polynomial ring $\bigoplus_{d\geq 0} (\operatorname{Sym}^d \chi_{\rm std})^{W}\cong \mathbb{C}[p_1,p_2,\ldots,p_n]$,
with $p_i$ homogeneous of degree $d_i$.
The numbers $d_1\leq d_2\leq \cdots \leq d_n$ are classified, and
the statement follows from the fact $d_1=2$ and $d_2\geq 4$ except type $A$, see \cite[Section 3.7]{Hum}.
\end{proof}

\begin{thm}\label{thm: type D}
    For a smooth hyperplane section $Y$ of a coadjoint Grassmannian $X$ of type $D$, the quantum cohomology $QH^*(Y)$ is non-semisimple.
\end{thm}
\begin{proof}
By Proposition \ref{prop: monodromy}, the symmetric cubic form
$$\big( \gamma_1\star_{q_Y=1}\gamma_2,\gamma_3\big)_{Y}\in \mathbb{C},\qquad \gamma_1,\gamma_2,\gamma_3\in H^\bullet(Y)$$
is $W$-invariant.
By Lemma \ref{len:Symstd}, when $\gamma_1,\gamma_2,\gamma_3\in H^*(Y)_{\rm std}$, we have
$( \gamma_1\star_{q_Y=1}\gamma_2,\gamma_3)_{Y} = 0$.
Since the Poincar\'e pairing is non-degenerate and $W$-invariant,
so is its restriction to $H^*(Y)_{\rm std}$ by Lemma \ref{len:Symstd}.
So we have $\gamma_1\star_{q_Y=1}\gamma_2\in H^*(Y)_{\rm inv}$.
 
If $\gamma_1,\gamma_2\in H^*(Y)_{\rm inv}$ and $\gamma_3\in H^*(Y)_{\rm std}$, then we have $(\gamma_1\star_{q_Y=1} \gamma_3,\gamma_2)_Y=(\gamma_1\star_{q_Y=1} \gamma_2,\gamma_3)_Y=0$.
Again, since the Poincar\'e pairing is non-degenerate and $W$-invariant, so is its restriction to $H^*(Y)_{\rm inv}$.
So we have $\gamma_1\star_{q_Y=1} \gamma_3\in H^*(Y)_{\rm std}$.

Notice that $\dim Y=2r_Y$, and $r_Y$ is even. Moreover,
$$H^{2i}(Y) = H^{2i}(Y)_{\rm inv} \text{ unless }i=r_Y.$$
By the discussion above, we can construct a subalgebra $\mathcal{B}$ of $\mathcal{A}$ with a $\mathbb{Z}_2$-grading by
\begin{align*}
\mathcal{B}^{\overline{0}}
&  
= H^0(Y)\oplus H^{2r_Y}(Y)_{\rm inv}\oplus H^{4r_Y}(Y),\\
\mathcal{B}^{\overline{1}}
& 
= H^{r_Y}(Y)\oplus H^{3r_Y}(Y)\oplus H^{2r_Y}(Y)_{\rm std}.
\end{align*}

Since
\begin{align*}
\dim H^{2r_Y}(Y)_{\rm std}
= \dim H^{2r_Y}(Y)_{\rm inv} & =\operatorname{rank}G,\\
\dim H^{0}(Y)
= \dim H^{4r_Y}(Y) & =1,\\
\dim H^{r_Y}(Y)=\dim H^{3r_Y}(Y) & >1,
\end{align*}
we have $\dim\mathcal{B}^{\overline{0}}
<\dim\mathcal{B}^{\overline{1}}$.
Hence, $\mathcal{B}$ contains a nonzero nilpotent element by Lemma \ref{lemma: key}.
\end{proof}

\subsection{Proof of Proposition \ref{cor:ofconj}}
By Proposition \ref{prop:hodge}, it suffices to determine  the semisimplicity for those smooth hyperplane sections $Y$ whose cohomology is  of Hodge-Tate type. By the hypothesis, we have assumed that Conjecture \ref{conj: ssQHY} holds. Each of its   cases has been distinguished accordingly in the statement to prove. Indeed, case (a) follows from Theorem \ref{mainthm:ttypeA} (2). Case (b) is a combination of Theorem \ref{mainthm: adcoad} with the fact that $C_n/P_1\cong \mathbb{P}^{2n-1}$ and $B_n/P_1$ (resp. $D_n/P_1$) is an odd (resp. even) dimensional smooth quadric, so that their smooth hyperplane sections   have  generic semisimple quantum cohomology, for being projective spaces or smooth quadrics as well. 

It remains to discuss case (c) in Conjecture \ref{conj: ssQHY}.
By direct computations using the Betti numbers $b_j(X)$ of $X$, we obtain the following table. For instance for $X=D_6/P_6$ with $i=4$, we have $\dim Y=\dim X-1=14$, $r_Y=r_X-1=9$ and $\dim \mathcal{A}^{\bar i}(Y)=b_8(Y)+b_{26}(Y)=b_8(Y)+b_{2}(Y)=b_8(X)+b_{2}(X)=2+1=3$, where the second equality uses the Poincar\'e duality and the third equality uses the Lefschetz theorem. Here  $\dim X, r_X$ and $b_j(X)$ are known and for instance can be directly read off from \cite{Bel}. Thus by Theorem \ref{mainthm:ss} (2), all these four cases have non-semisimple small quantum cohomology. 
\[
\begin{array}{|c|c|c|c|c|c|c|}
  \hline
  X & \dim Y & r_Y & i  & \dim \mathcal{A}^{\bar i}(Y)& \dim \mathcal{A}^{-\bar i}(Y) \\
 \hline \hline
 D_6/P_6 & 14 & 9 &4 & b_8(X)+b_2(X)=3 & b_{10}(X)+b_0(X)=4\\ \hline
  D_7/P_7 & 20 & 11 &  2 & b_4(X)+b_{14}(X)=5 &b_{18}(X)+b_0(X)=6\\ \hline
   E_6/P_1 & 15 & 11 & 4 & b_8(X)+b_0(X)=3 & b_{14}(X)=2\\ \hline
    E_7/P_7 & 26 & 17 & 8 & b_{16}(X)+b_2(X)=3 &b_{18}(X)+b_0(X)=4\\ \hline 
\end{array}
\] 

Note that $D_4/P_4$ is a smooth 6-dimensional smooth quadric. Hence its general hyperplane section $Y$ is isomorphic to a $5$-dimensional smooth quadric and  has generic semisimple small quantum cohomology. The Schubert classes for    $C_3/P_3$ and $D_5/P_5\cong B_4/P_4$  can both be indexed by strict partitions. Their general hyperplane sections $Y$ satisfy $\dim Y=5, r_Y=3$ and $\dim Y=9, r_Y=7$ respectively. Consequently, the algebras $\mathcal{A}^0(Y)$ are both $2$-dimensional, generated by $\{1, j^*\sigma_3\}$ and $\{1, j^*\sigma_{43}\}$ respectively. 
 Proposition~\ref{prop: qPieriY_deg1}
applies to the quantum multiplication by $h_Y$ in both cases, with the same proof.

For $X=C_3/P_3$,  
Proposition~\ref{prop: deg2GW} also gives
$\langle h_Y, P.D.[\mathrm{pt}],P.D.[\mathrm{pt}]\rangle_2^Y=0$, by similar span arguments, which is the only  possible quantum term  of degree $>1$ in the  multiplication by $h_Y$. By direct calculations using the quantum Chevalley formula \cite{FuWo},  we obtain
\[
 h_Y^3=8j^*\sigma_3+4q_Y,\qquad
 h_Y^3\star j^*\sigma_3=20q_Yj^*\sigma_3+8q_Y^2.
\]

For $X=D_5/P_5\simeq B_4/P_4$, there are no quantum terms of degree $>1$ in the multiplication by $h_Y$, simply by the dimension constraint. By direct calculations, we have  
\[ h_Y^7=12j^*\sigma_{43}+5q_Y,\qquad
 h_Y^7\star j^*\sigma_{43}=29q_Yj^*\sigma_{43}+12q_Y^2.
\]
At $q_Y=1$, these identities show that $h_Y^{r_Y}$ generates
$\mathcal A^0(Y)$, with respective minimal polynomials
\[
 t^2-24t+16,\qquad t^2-34t+1.
\]
Both have distinct nonzero roots, so
$\mathcal A^0(Y)\simeq\mathbb C\oplus\mathbb C$ and $h_Y$ is invertible.
 Therefore $QH^*(Y)$ is generically
semisimple by Lemma \ref{lemma: AtoAy}.  This completes the proof of Proposition \ref{cor:ofconj}.

\makeatletter
\@setaddresses
\def\@setaddresses{}
\makeatother

\appendix
\section{Generic non-semisimplicity of small quantum cohomology of Kronecker moduli \\[6pt] by Pieter Belmans and Markus Reineke}\label{sec:appendix}
In this appendix,
we explain how the semisimplicity obstruction
from Theorem~\ref{mainthm:ss}
is strong enough to show that certain Kronecker quiver moduli
have non-semisimple small quantum cohomology.
Quiver moduli are moduli spaces of (semi)stable representations of a quiver,
and they admit many strong results,
see \cite{MR2484736} for a survey.
We will focus on the case where the quiver is the~$m$-Kronecker quiver with~$m$ arrows:
\begin{equation}
  \begin{tikzpicture}[baseline=(current bounding box.center), vertex/.style={circle, draw, inner sep=1.6pt}]
    \node[vertex] (v1) at (0,0) {};
    \node[vertex] (v2) at (2,0) {};
    \draw[->] (v1) edge [bend left]  (v2);
    \draw[->] (v1) edge [bend right] (v2);
    \node at (1,.1) {$\vdots$};
  \end{tikzpicture}\,,
\end{equation}
where we will throughout assume that~$m\geq 3$.
We will write~$(d,e)$ for the dimension vector,
where~$d,e\geq 1$.
We will assume that~$\gcd(d,e)=1$.
As there are only~2~vertices,
there is a unique non-trivial stability function,
for which stability coincides with semistability.
We denote the moduli space of (semi)stable representations of the~$m$-Kronecker quiver
with dimension vector~$(d,e)$
by~$\kronecker{m}{d}{e}$.

The important properties for us are that
\begin{enumerate}
  \item $\kronecker{m}{1}{e}\cong\Gr(e,m)$ and~$\kronecker{m}{d}{1}\cong\Gr(d,m)$;
  \item $\kronecker{m}{d}{e}$ is Hodge--Tate \cite[Theorem~3]{MR1324213};
  \item $\kronecker{m}{d}{e}$ is a smooth projective rational Fano variety
    of dimension~$mde-d^2-e^2+1$,
    Picard rank~1,
    and index~$m$,
    see \cite[Corollary~5.2]{MR4352662}.
\end{enumerate}

In addition to generic semisimplicity of small quantum cohomology of
Kronecker moduli which happen to be Grassmannians,
we have the following important result \cite[Theorem~5.23]{2412.15987v1},
which follows from an explicit presentation of the small quantum cohomology.
\begin{proposition}[Meng]
  \label{proposition:meng}
  The small quantum cohomology~$\QH(\kronecker{3}{2}{3})$ is generically semisimple.
\end{proposition}
Given that a version of Schofield's conjecture predicts that
the derived category of a quiver moduli space
(and thus~$\kronecker{m}{d}{e}$ in particular)
admits a full exceptional collection \cite{MR4954467},
and Dubrovin's conjecture then predicts that its big quantum cohomology
is generically semisimple,
one might (optimistically) wonder whether the small quantum cohomology is already semisimple,
as is the case in all known examples.
However,
the following shows that this is \emph{not} the case.
\begin{theorem}
  \label{theorem:non-semisimple}
  Assume that~$m\geq 5$ is odd.
  Then the small quantum cohomology~$\QH(\kronecker{m}{2}{m})$ is not generically semisimple.
\end{theorem}
We will obtain this as a consequence of Theorem~\ref{mainthm:ss},
and thus we can conclude that
this result is also strong enough to show that certain Kronecker moduli spaces
have generically non-semisimple small quantum cohomology,
similar to the application in Theorem~\ref{thm:GrEF}
for exceptional generalized Grassmannians.

\paragraph{Two lemmas}
Given a polynomial with integer coefficients
\begin{equation}
  P(q)=\sum_{i=0}^Na_iq^i\in\mathbb{Z}[q]
\end{equation}
and an integer~$n\geq 1$,
we define
\begin{equation}
  \tilde{b}_j\colonequals\sum_{k\equiv j\bmod n}\!\!a_k
\end{equation}
for $j\in\mathbb{Z}/n\mathbb{Z}$.
Later on,
the polynomial will be the even Poincar\'e polynomial of a smooth projective variety,
and the~$\tilde{b}_j$ will be the index-periodic even Betti numbers.

\begin{lemma}
  \label{lemma:integer-valued}
  Assume that $\tilde{b}_j\leq \tilde{b}_{jk}$ for all $j,k\in\mathbb{Z}/n\mathbb{Z}$.
  Then $P(\ee^{2\pi\ii/n})=\sum_{d|n}\mu(n/d)\tilde{b}_d$,
  in particular, it is an integer.
\end{lemma}

\begin{proof}
  By assumption we have that $\tilde{b}_j=\tilde{b}_{jk}$
  if~$k\in(\mathbb{Z}/n\mathbb{Z})^\times$.
  The orbits of $(\mathbb{Z}/n\mathbb{Z})^\times$ on~$\mathbb{Z}/n\mathbb{Z}$ correspond to the divisors of $n$,
  thus we can write
  \begin{equation}
    P(\ee^{2\pi\ii/n})
    =\sum_{j=1}^n\tilde{b}_j\ee^{2\pi\ii j/n}
    =\sum_{d|n}\tilde{b}_d\sum_{(j,n)=d}\ee^{2\pi\ii j/n}
    =\sum_{d|n}\tilde{b}_d\sum_{(j,n/d)=1}\ee^{2\pi\ii j/(n/d)}
    =\sum_{d|n}\mu(n/d)\tilde{b}_d
  \end{equation}
  where the last step uses standard properties of Ramanujam sums \cite[Theorem 271]{MR2445243},
  and the latter expression is indeed an integer.
\end{proof}

We also recall the $q$-Lucas theorem on evaluations of Gaussian binomial coefficients $\qbinom{a}{b}_q$
at roots of unity \cite[Proposition 2.2]{MR0656008}.
\begin{lemma}
  \label{lemma:q-lucas}
  For $a,b\in\mathbb{N}$ and $0\leq r,s<n$,
  we have
  \begin{equation}
    \qbinom{an+r}{bn+s}_{\ee^{2\pi\ii/n}}
    =
    \binom{a}{b}\cdot\qbinom{r}{s}_{\ee^{2\pi\ii/n}}.
  \end{equation}
\end{lemma}

\paragraph{Poincar\'e polynomials of Kronecker moduli}
Now we consider the Kronecker moduli spaces~$\kronecker{m}{2}{e}$
for $e$ odd and $m,e\geq 3$.
There is a duality
\begin{equation}
  \kronecker{m}{2}{e}
  \cong
  \kronecker{m}{2}{2m-e},
\end{equation}
(with both being empty when~$e>2m$)
see, e.g., \cite[Corollary 4.1]{epiga:14742},
thus we can assume $e\leq m$.

Their Poincar\'e polynomials admit a closed expression using a resolved Harder-Narasimhan type recursion \cite[Section 7]{MR1974891}.
\begin{lemma}
  \label{lemma:harder-narasimhan-recursion}
  Let~$e,m\geq 3$,
  and assume that $e\leq m$, and~$e$ is odd.
  The Poincar\'e polynomial $\poincare{m}{2}{e}(q)$ of $\kronecker{m}{2}{e}$ is given by
  \begin{equation}
    \poincare{m}{2}{e}(q)
    =
    \frac{1}{q(q-1)}\left(\frac{1}{q+1}\qbinom{2m}{e}_q-\sum_{k=0}^{(e-1)/2}q^{(m-e+k)k}\qbinom{m}{k}_q\qbinom{m}{e-k}_q\right).
  \end{equation}
\end{lemma}

We can now derive the main observation.

\begin{lemma}
  \label{lemma:strictly-imaginary}
  Let~$e,m\geq 3$,
  and assume that $e\leq m$, and~$e$ is odd.
  Then $\poincare{m}{2}{e}(\ee^{2\pi\ii/m})$ is strictly imaginary
  if and only if $e=m\geq 5$.
\end{lemma}

\begin{proof}
  If $e<m$,
  then all Gaussian binomial coefficients in Lemma~\ref{lemma:harder-narasimhan-recursion}
  evaluate to zero at $q=\ee^{2\pi\ii/m}$ by Lemma~\ref{lemma:q-lucas},
  because we always have~$r=0$ and~$s=e\geq 1$,
  and thus~$\poincare{m}{2}{e}(\ee^{2\pi\ii/m})=0$.

  If $e=m$,
  we find
  \begin{equation}
    \begin{aligned}
      \poincare{m}{2}{m}(\ee^{2\pi\ii/m})
      &=\frac{1}{\ee^{2\pi\ii/m}(\ee^{2\pi\ii/m}-1)}\left(\frac{1}{\ee^{2\pi\ii/m}+1}\cdot 2-1\right) \\
      &=-\frac{1}{\ee^{2\pi\ii/m}(\ee^{2\pi\ii/m}+1)},
    \end{aligned}
  \end{equation}
  as for~$k=0$ in the summation in Lemma~\ref{lemma:harder-narasimhan-recursion}
  we have~$\qbinom{m}{0}_{\ee^{2\pi\ii/m}}=\qbinom{m}{m}_{\ee^{2\pi\ii/m}}=1$ by definition,
  whilst all other Gaussian binomial coefficients in the sum
  evaluate to zero at $q=\ee^{2\pi\ii/m}$ by Lemma~\ref{lemma:q-lucas}
  because for~$k\geq 1$ we will have~$r=0$ and~$s\geq 1$,
  whereas~$\qbinom{2m}{m}_{\ee^{2\pi\ii/m}}=\binom{2}{1}\qbinom{0}{0}_{\ee^{2\pi\ii/m}}=2$
  by another application of Lemma~\ref{lemma:q-lucas}.
  This is strictly imaginary for $m\geq 5$.
\end{proof}

\begin{proof}[Proof of Theorem \ref{theorem:non-semisimple}]
  By Lemma~\ref{lemma:strictly-imaginary}
  we have that~$\poincare{m}{2}{m}(\ee^{2\pi\ii/m})$ is strictly imaginary (and thus nonzero).
  By Lemma~\ref{lemma:integer-valued}, this invalidates the criterion of Theorem~\ref{mainthm:ss},
  and thus the small quantum cohomology of~$\kronecker{m}{2}{m}$ is not generically semisimple.
\end{proof}

Let us illustrate what happens in the smallest example, where~$m=5$.
\begin{example}
  The Kronecker moduli space~$\kronecker{5}{2}{5}$ is a~22-dimensional smooth projective Fano variety,
  of index~5.
  Using the algorithm for Betti numbers of \cite[Corollary~6.9]{MR1974891},
  and implemented in~\cite{quivertools},
  or by working out the special case from Lemma~\ref{lemma:harder-narasimhan-recursion},
  we obtain that the even Betti numbers are
  \begin{equation}
    \label{equation:even-betti}
    1, 1, 3, 4, 8, 11, 17, 22, 30, 35, 41, 41, 41, 35, 30, 22, 17, 11, 8, 4, 3, 1, 1.
  \end{equation}
  Hence,
  the index-periodic even Betti numbers~$\tilde{\mathrm{b}}_0,\ldots,\tilde{\mathrm{b}}_4$
  are
  \begin{equation}
    78, 77, 78, 77, 77.
  \end{equation}
  We see that~$\tilde{\mathrm{b}}_2>\tilde{\mathrm{b}}_4$,
  hence by Theorem~\ref{mainthm:ss}
  we obtain that~$\QH(\kronecker{5}{2}{5})$ cannot be generically semisimple.
\end{example}

\paragraph{Acknowledgements}
P.~B.~was supported by
NWO (Dutch Research Council)
as part of the grant
\href{https://doi.org/10.61686/RZKLF82806}{\texttt{doi:10.61686/RZKLF82806}}.
M.~R.~thanks
University D\"usseldorf for the kind hospitality.
We would like to thank Maxim Smirnov
for interesting discussions,
and for asking a question about horospherical varieties,
which led us to consider the case of Kronecker moduli.

{\ } \\
\footnotesize
\indent\textsc{Mathematical Institute, Utrecht University, Budapestlaan 6, 3584CD Utrecht, Netherlands} \\
\indent\emph{Email address}: \url{p.belmans@uu.nl} \\

\textsc{Faculty of Mathematics, Ruhr-Universit\"at Bochum, Universit\"atsstra\ss e 150, 44780 Bochum, Germany} \\
\indent\emph{Email address}: \url{markus.reineke@ruhr-uni-bochum.de}

\end{document}